\documentclass[a4paper]{amsart}
\usepackage{amssymb,latexsym}
\usepackage[utf8x]{inputenc}
\usepackage{amsmath,amsthm}
\usepackage{amsfonts,mathrsfs}
\usepackage{hyperref}
\usepackage{cleveref}
\usepackage{enumerate,units}
\usepackage[all]{xy}
\usepackage{graphicx,url}

\theoremstyle{plain}
\newtheorem{theorem}{Theorem}[section]
\newtheorem{corollary}[theorem]{Corollary}
\newtheorem{lemma}[theorem]{Lemma}
\newtheorem{proposition}[theorem]{Proposition}
\newtheorem{conclusion}[theorem]{Conclusion}
\newtheorem*{proposition*}{Proposition}

\newtheorem{fact}[theorem]{Fact}
\newtheorem{claim}{Claim}[theorem]
\newtheorem*{subclaim}{Subclaim}
\newtheorem*{theorem*}{Theorem}
\newtheorem*{context*}{Assumption $\diamondsuit$}

\theoremstyle{definition}
\newtheorem{definition}[theorem]{Definition}
\newtheorem{assumption}[theorem]{Assumption}
\newtheorem{example}[theorem]{Example}

\theoremstyle{remark}
\newtheorem{remark}[theorem]{Remark}

\newtheorem*{conjecture}{\textbf{Conjecture}}

\numberwithin{equation}{section}
\newcommand{\forkindep}[1][]{%
  \mathrel{
    \mathop{
      \vcenter{
        \hbox{\oalign{\noalign{\kern-.3ex}\hfil$\vert$\hfil\cr
              \noalign{\kern-.7ex}
              $\smile$\cr\noalign{\kern-.3ex}}}
      }
    }\displaylimits_{#1}
  }
}

\newenvironment{claimproof}[1][\proofname]
  {%
    \proof[#1]%
  }
  {%
    \endproof%
  }

\newenvironment{subclaimproof}[1][\proofname]
  {%
    \proof[#1]%
  }
  {%
    \endproof%
  }

\newcounter{step}                   
    {\hfill $\clubsuit$             
     \vspace{7pt}\par}

\makeatletter
\providecommand*{\cupdot}{%
  \mathbin{%
    \mathpalette\@cupdot{}%
  }%
}
\newcommand*{\@cupdot}[2]{%
  \ooalign{%
    $\m@th#1\cup$\cr
    \hidewidth$\m@th#1\cdot$\hidewidth
  }%
}
\makeatother

\newcommand{\dom}{\mathrm{Dom}}
\newcommand{\U}{\underline}
\newcommand{\Ba}{\textbf{a}}
\newcommand{\Bb}{\textbf{b}}
\newcommand{\Bc}{\textbf{c}}
\newcommand{\kap}{\varkappa}
\newcommand{\Sh}{\text{Sh}}
\newcommand{\LSh}{\text{LSh}}
\newcommand{\RSh}{\text{RSh}}

\newcommand{\E}{\mathrel{E}}
\newcommand{\R}{\mathrel{R}}
\newcommand{\D}{\mathrel{D}}
\newcommand{\suc}[1]{\mathrm{Suc}(#1)}

\DeclareMathOperator{\acl}{acl}
\DeclareMathOperator{\dcl}{dcl} 
\DeclareMathOperator{\tp}{tp}
\DeclareMathOperator{\otp}{otp}

\DeclareMathOperator{\Rg}{Range}
\DeclareMathOperator{\Img}{Im}
\DeclareMathOperator{\cf}{cf}

\begin{document}
\title{Infinite Stable Graphs With Large Chromatic Number II}
\begin{abstract}
We prove a version of the strong Taylor's conjecture for stable graphs: if $G$ is a stable graph whose chromatic number is strictly greater than $\beth_2(\aleph_0)$ then $G$ contains all finite subgraphs of $Sh_n(\omega)$ and thus has elementary extensions of unbounded chromatic number. This completes the picture from our previous work. The main new model theoretic ingredient is a generalization of the classical construction of Ehrenfeucht-Mostowski models to an infinitary setting, giving a new characterization of stability.
\end{abstract}

\author{Yatir Halevi}
\address{Department of Mathematics, Ben Gurion University of the Negev, Be'er-Sheva 84105, Israel.}
\email{yatirbe@post.bgu.ac.il}

\author{Itay Kaplan}
\address{Einstein Institute of Mathematics, Hebrew University of Jerusalem, 91904, Jerusalem
Israel.}
\email{kaplan@math.huji.ac.il}

\author{Saharon Shelah}
\address{Einstein Institute of Mathematics, Hebrew University of Jerusalem, 91904, Jerusalem
Israel.}
\email{shelah@math.huji.ac.il}

\thanks{The first author would like to thanks the Israel Science Foundation for its support of this research (grant No. 181/16) and the Kreitman foundation fellowship. The second author would like to thank the Israel Science Foundation for their support of this research (grant no. 1254/18). The third author would like to thank the Israel Science Foundation grant no. 1838/19. This is Paper no. 1211 in the third author's  publication list.}


\maketitle

\section{Introduction}
The chromatic number $\chi(G)$ of a graph $G=(V,E)$ is the minimal cardinal $\kap$ for which there exists a vertex coloring with $\kap$ colors. There is a long history of structure theorems deriving from large chromatic number assumptions, see e.g., \cite{komjath}. The main topic of this paper will be the following conjecture proposed by Erd\"os-Hajnal-Shelah \cite[Problem 2]{EHS} and Taylor \cite[Problem 43, page 508]{Taylorprob43}.
\begin{conjecture}[Strong Taylor's Conjecture]
For any graph $G$ with $\chi(G)>\aleph_0$ there exists an $n\in\mathbb{N}$ such that $G$ contains all finite subgraphs of $\Sh_n(\omega)$.
\end{conjecture}
Where, for a caridnal $\kappa$, the shift graph $\Sh_n(\kappa)$ is the graph whose vertices are increasing $n$-tuples of ordinals less than $\kappa$, and we put an edge between $s$ and $t$ if for every $1\leq i\leq n-1$, $s(i)=t(i-1)$ or vice-versa. The shift graphs $\Sh_n(\kappa)$ have large chromatic numbers depending on $\kappa$, see Fact \ref{F:Sh-high chrom} below. Consequently, if the strong Taylor's conjecture holds for a graph $G$, it has elementary extensions of unbounded chromatic number (having the same family of finite subgraphs). 

The strong Taylor's conjecture was refuted in \cite[Theorem 4]{HK}. See \cite{komjath} and the introduction of \cite{1196} for more historical information.

In \cite{1196} we initiated the study of variants of the strong Taylor's conjecture for some classes of graphs with \emph{stable} first order theory (stable graphs). Stablility theory, which is the study of stable theories and originated in the works of the third author in the 60s and 70s, is one of the most influential and important subjects in modern model theory. Examples of stable theories include abelian groups, modules, algebraically closed fields, graph theoretic trees, or more generally superflat graphs \cite{PZ}. Stablility also had an impact in combinatorics, e.g. \cite{MS} and \cite{CPT} to name a few.

More precisely, in \cite{1196} we proved the strong Taylor's conjecture for $\omega$-stable graphs and variants of the conjecture for superstable graphs (replacing $\aleph_0$ by $2^{\aleph_0}$) and for stable graphs which are interpretable in a stationary stable theory (replacing $\aleph_0$ by $\beth_2(\aleph_0)$). As there exist stable graphs that are not interpretable in a stationary stable structure, see \cite[Proposition 5.22, Remark 5.23]{1196}, we asked what is the situation in general stable graphs and in this paper we answer it with the following theorem.

\begin{theorem*}[Corollary \ref{C:main corollary}]
Let $G=(V,E)$ be a stable graph. If $\chi(G)>\beth_2(\aleph_0)$ then $G$ contains all finite subgraphs of $\Sh_n(\omega)$ for some $n\in \mathbb{N}$.
\end{theorem*}

The key tool in proving the results for $\omega$-stable graphs and superstable graphs is that every large enough saturated model is an Ehrenfeucht–Mostowski model (EM-model) in some bounded expansion of the language. 

An EM-model is a model which is the definable closure of an indiscernible sequence and was originally used by Ehrenfeucht–Mostowski in order to find models with many automorphisms \cite{EM}. It was shown by Lascar \cite[Section 5.1]{lascar} that every saturated model of cardinality $\aleph_1$ in an $\omega$-stable theory is an EM-model in some countable expansion of the language, and was later generalized to any cardinality by Mariou \cite[Theorem C]{mariou}. This was  generalized for superstable theories by Mariou \cite[Theorem 3.B]{mariouthesis} and in an unpublished preprint by the third author \cite{Sh:1151}.

It was shown by Mariou \cite[Theorem 3.A]{mariouthesis} that in a certain sense the existence of such saturated EM-models for a stable theory necessarily implies that the theory is superstable. Consequently a different tool is needed in order to prove the theorem for general stable theories. 

In the stationary stable case, we use a variant of representations of structures in the sense of \cite{919}. However, this method did not seem to easily adjust to the general stable case.

In this paper we resolved this problem by generalizing the notion of EM-models to \emph{infinitary EM-models} and show that such saturated models exist for any stable theory in Theorem \ref{T:existence of gen em model in stable}. The definition is a bit technical, so here we will settle with an informal description: 

In an EM-model every element is given by a term and a finite sequence of elements from the generating indiscernible sequence. Analogously, in an infinitary EM-model every element is given by some ``term'' with infinite (but bounded) arity and a suitable sequence of elements from an indiscernible sequence. 

We prove that the existence of saturated infinitary EM-models characterizes stability.

\begin{theorem*}[Theorem \ref{T:existence of gen em model in stable}]
The following are equivalent for a complete $\mathcal{L}$-theory $T$:
\begin{enumerate}
\item $T$ is stable.

\item Let $\kappa,\mu$ and $\lambda$ be cardinals satisfying $\kappa=\cf(\kappa)\geq \kappa(T)+\aleph_1$, $\mu^{<\kappa}=\mu\geq 2^{\kappa+|T|}$ and $\lambda=\lambda^{<\kappa}\geq \mu$ and let $T\subseteq T^{sk}$ be an expansion with definable Skolem functions such that $|T|=|T^{sk}|$ in a language $\mathcal{L}\subseteq \mathcal{L}^{sk}$. Then there exists an infinitary EM-model  $M^{sk}\models T^{sk}$ based on $(\alpha,\lambda)$, where $\alpha\in \kappa^U$ for some set $U$ of cardinality at most $\mu$, such that $M=M^{sk}\restriction \mathcal{L}$ is saturated of cardinality $\lambda$.
\end{enumerate}
\end{theorem*}

Section \ref{s:em-models} is the only purely model theoretic section and is the only place where stability is used. The results of this section (more specifically Theorem \ref{T:existence of gen em model in stable}) are only used in Section \ref{s:conclusion}. In Section \ref{S:order type graphs} we study graphs on (perhaps infinite) increasing sequences whose edge relation is determined by the order type. Aiming to prove that if the chromatic number is large, then one can embed shift graphs, we analyze several different cases. The last case we deal with in Section \ref{S:order type graphs} turns out to be rather complicated, so we devote all of Section \ref{S:PCF} to it. There, we employ ideas inspired by PCF theory to get a coloring of small cardinality. Section \ref{s:conclusion} concludes.

\section{Preliminaries}
We use small latin letters $a,b,c$ for tuples and capital letters $A,B,C$ for sets. We also employ the standard model theoretic abuse of notation and write $a\in A$ even for tuples when the length of the tuple is immaterial or understood from context.

For any two sets $A$ and $J$, let $A^{\underline{J}}$ be the set of injective functions from $J$ to $A$ (where the notation is taken from the falling factorial notation), and if $(A,<)$ and $(J,<)$  are both linearly ordered sets, let $(A^{\U J})_<$ be the subset of $A^{\U J}$ consisting of strictly increasing functions. If we want to emphasize the order on $J$ we will write $(A^{\U{(J,<)}})_<$.  

Throughout this paper, we interchangeably use sequence notation and function notation for elements of $A^{\U J}$, e.g. for $f\in A^{\U J}$, $f(i)=f_i$. For any sequence $\eta$ we denote by $\Rg(\eta)$ the underlying set of the sequence (i.e. its image). If $(A,<^A)$ and $(B,<^B)$ are linearly ordered sets, then the most significant coordinate of the lexicographic order on $A\times B$ is the left one.

\subsection{Stability}
We use fairly standard model theoretic terminology and notation, see for example \cite{TZ,guidetonip}. We gather some of the needed notions. For stability, the reader can also consult with \cite{classification}.

We denote by $\tp(a/A)$ the complete type of $a$ over $A$. Let $(I,<)$ be a linearly ordered set. A sequence $\langle a_i: i\in I\rangle$ inside a first order structure is \emph{indiscernible} if for any $i_1<\dots<i_k$ and $j_1<\dots<j_k$ in $I$,
\[\tp(a_{i_1},\dots,a_{i_k})=\tp(a_{j_1},\dots, a_{j_k}).\]
A structure $M$ is \emph{$\kappa$-saturated}, for a cardinal $\kappa$, if any type $p$ over $A$ with $|A|< \kappa$ is realized in $M$. The structure $M$ is \emph{saturated} if it is $|M|$-saturated.
A \emph{monster model} for $T$, usually denoted by $\mathbb{U}$, is a large saturated model containing all sets and models (as elementary substructures) we will encounter\footnote{There are set theoretic issues in assuming that such a model exists, but these are overcome by standard techniques from set theory that ensure the generalized continuum hypothesis from some point on while fixing a fragment of the universe. The reader can just accept this or alternatively assume that $\mathbb{U}$ is merely $\kappa$-saturated and $\kappa$-strongly homogeneous for large enough $\kappa$.}.  All subsets and models will be \emph{small}, i.e. of cardinality $<|\mathbb{U}|$.

A first theory $T$ is \emph{stable} if there does not exist a model $M\models T$, a formula $\varphi(x,y)$ and elements $\langle a_i\in M: i<\omega\rangle$ such that $M\models \varphi(a_i,a_j)\iff i<j$. An equivalent definition is that there exists some $\kappa\geq |T|$ such that for all $M\models T$ with $|M|\leq\kappa$ the cardinality of complete types over $M$ is at most $\kappa$. For any such $\kappa$, $T$ has a saturated model of cardinality of $\kappa$ \cite[Theorem VIII.4.7]{classification}. 

Every indiscernible sequence in a stable theory is \emph{totally indiscernible}, i.e. in the notation above, for any $i_1,\dots,i_k$ and $j_1,\dots,j_k$ in $I$,
\[\tp(a_{i_1},\dots,a_{i_k})=\tp(a_{j_1},\dots, a_{j_k}).\]
Other than these notions we will also require basic understanding in forking. See the above references for more information.

\subsection{Graph theory}
Here we gather some facts on graphs and the chromatic number of graphs (all can be found in \cite{1196}).

By a \emph{graph} we mean a pair $G=(V,E)$ where $E\subseteq V^2$ is symmetric and irreflexive. A \emph{graph homomorphism} between $G_1=(V_1,E_1)$ and $G_2=(V_2,E_2)$ is a map $f:V_1\to V_2$ such that $f(e)\in E_2$ for every $e\in E_1$. If $f$ is injective we will say that $f$ embeds $G_1$ into $G_2$ a subgraph. If in addition we require that $f(e)\in E_2$ if and only if $e\in E_1$ we will say that $f$ embeds $G_1$ into $G_2$ as an induced subgraph.

\begin{definition}
Let $G=(V,E)$ be a graph.
\begin{enumerate}
\item For a cardinal $\kap$, a \emph{vertex coloring} (or just coloring) of cardinality $\kap$ is a function $c:V\to \kap$ such that $x\E y$ implies $c(x)\neq c(y)$ for all $x,y\in V$.
\item The \emph{chromatic number} $\chi(G)$ is the minimal cardinality of a vertex coloring of $G$.
\end{enumerate}
\end{definition}

These are the basic properties of $\chi(G)$ that we will require.

\begin{fact}\cite[Lemma 2.3]{1196}\label{F:basic prop}
Let $G=(V,\E)$ be a graph.
\begin{enumerate}

\item If $V=\bigcup_{i\in I} V_i$ then $\chi(G)\leq \sum_{i\in I} \chi(V_i, E\restriction V_i)$.
\item If $E=\bigcup_{i\in I} \E_i$ (with the $E_i$ being symmetric)  then $\chi(G)\leq \prod_{i\in I}\chi(V,E_i)$.
\item If $\varphi:H\to G$ is a graph homomorphism then $\chi(H)\leq \chi(G)$.
\item If $\varphi:(H,E^H)\to (G,E^G)$ is a surjective graph homomorphism with $e\in E^H\iff \varphi(e)\in E^G$ then $\chi(H)=\chi(G)$.
\end{enumerate}
\end{fact}

\begin{example} For any finite number $1\leq r$ and any linearly ordered set $(A,<)$, let $\Sh_r(A)$, or $\Sh_r(A,<)$ if we want to emphasize the order, (\emph{the shift graph on $A$}) be the following graph: its set of vertices is the set $(A^{\U r})_<$ of increasing $r$-tuples, $s_0,\dots,s_{r-1}$, and we put an edge between $s$ and $t$ if for every $1\leq i\leq r-1$, $s(i)=t(i-1)$, or vice-versa. It is an easy exercise to show that $\Sh_r(A)$ is a connected graph. If $r=1$ this gives $K_{A}$, the complete graph on $A$.
\end{example}

\begin{fact}\cite[Fact 2.6]{1196}\cite[Proof of Theorem 2]{EH-shift}\label{F:Sh-high chrom}
Let $2\leq r<\omega$ be a natural number and $\kap$ be a cardinal, \[\chi\left(\Sh_r(\beth_{r-1}\left(\kap\right)^{+})\right)\geq\kap^{+}.\]
\end{fact}

Finally, the following fact is a very useful tool.
\begin{fact}\cite[Proposition 3.2]{1196}\label{F:homomorphism is enough}
Let $G=(V,E)$ be a graph and assume there exists an homomorphism of graphs $t:\Sh_k(\omega)\to G$. Then there exists $n\leq k$, such that 
\begin{itemize}
\item[($\dagger$)] $G$ contains all finite subgraphs of $\Sh_n(\omega)$.
\end{itemize}

Consequently, if $H$ is a graph that contains all finite subgraphs of $\Sh_k(\omega)$, for some $k$, and $t:H\to G$ is a homomorphism of graphs, then there exists some $n\leq k$ such that $G$ satisfies ($\dagger$).
\end{fact}

\section{Infinitary EM-models and stability}\label{s:em-models}
Let $T$ be a first order theory and $\mathbb{U}$ a monster model for $T$.

An \emph{EM-Model} for $T$ is a model that is the definable closure of an indiscernible sequence (possibly in some expansion of the theory which admits Skolem functions).

Every element in an EM-model is of the form $t(a_{i_1},\dots, a_{i_n})$, where $t$ is a term (in the expanded language) and $a_{i_1},\dots, a_{i_n}$ are elements of the indiscernible sequence. In other words, to any element we may associate a pair $(i,\eta)$, where $i<|T|$ (this codes the term $t_i(\bar x_i)$) and $\eta$ is an increasing sequence of cardinality $|\bar x_i|$. 

Mariou \cite{mariouthesis,mariou} and Shelah \cite{Sh:1151} proved that if $T$ is $\omega$-stable or even superstable then it has an EM-model in some expansion of the language whose restriction to the original language is saturated. For general stable theories, as we will see in this section, one needs to allow ``terms'' with, possibly, infinite arity to get a parallel result.
 
Let $\kappa\geq \aleph_0$ be a regular cardinal (which we think of as a bound on the arity) and let $\mu$ be a cardinal (which we think of as a bound on the number of terms). Let $\alpha\in \kappa^\mu$ be a function assigning to each function its arity.

%
%
%

\begin{definition}\label{D:alpha-I}
Let $\kappa$ a cardinal, $(I,<)$ a linearly ordered set, $U$ a set and $\alpha\in \kappa^U$.
Let $a=\langle a_{i,\eta}:i\in U ,\,\eta\in (I^{\U{\alpha_i}})_<\rangle$ be a sequence of tuples from $\mathbb{U}$.

We say that $a$ is \emph{$(\alpha,I)$-indiscernible} if for every  $\langle i_j\in U:j<k\rangle$,  $\langle \eta_j\in (I^{\U{\alpha_{i_j}}})_<:j<k\rangle$ and $\langle \rho_j\in (I^{\U{\alpha_{i_j}}})_<:j<k\rangle$ if
there exists a partial isomorphism of $(I,<)$ mapping $\langle \eta_j:j<k\rangle$ to $\langle \rho_j:j<k\rangle$ then $\langle a_{i_j,\eta_j}:j<k\rangle$ and $\langle a_{i_j,\rho_j}:j<k\rangle$ have the same type.
\end{definition}
%

Recall that given a subset $A\subseteq \mathbb{U}$ and an ultrafilter $\mathcal{D}$ on $A$ we may define the global average type $p_\mathcal{D}=\mathrm{Av}(\mathcal{D},\mathbb{U})$ by
\[p_\mathcal{D}\vdash \varphi(x,b)\iff \varphi(A,b)\in \mathcal{D}.\] Obviously, $p_\mathcal{D}$ is finitely satisfiable in $A$. 
\begin{remark}\label{R:global-ultrafilter}
If $\mathcal{D}$ is an ultrafilter on $A$ and $A\subseteq B$ then $\{U\in B: \exists V\in \mathcal{D} (V\subseteq U)\}$ is the a unique ultrafilter $\mathcal{D}^\prime$ on $B$ containing $\mathcal{D}$ and $p_{\mathcal{D}}=p_{\mathcal{D}^\prime}$. 

\end{remark}

For any linearly ordered set $(I,<)$ and $A\subseteq B$, we say that $\langle a_i:i\in I\rangle$ realizes $(p_{\mathcal{D}})^{\otimes I}|B$ if for any $k\in I$, $a_k\models p_{\mathcal{D}}|B\langle a_i :i<k\in I\rangle$.

\begin{proposition}\label{P:existence-of-indisc}
Assume that $\mathbb{U}$ has definable Skolem functions. Let $\kappa\geq \aleph_0$ a regular cardinal and $\mu^{<\kappa}=\mu\geq 2^{\kappa+|\mathcal{L}|}$ a cardinal.

Let $\alpha\in \kappa^\mu$ be any function and let $(I,<)$ be any infinite linear order. 
\begin{enumerate}
\item  There exist $U\subseteq \mu\times \mu$ and an $(\alpha^\prime,I)$-indiscernible sequence \[a=\langle a_{i,k,\eta}: (i,k)\in U,\, \eta\in (I^{\U{\alpha_i}})_<\rangle,\] where $\alpha^\prime\in\kappa^U$ is defined by $\alpha^\prime_{(i,k)}=\alpha_i$, for $(i,k)\in U$, such that $\mathbb{U}\restriction \dcl(\Rg (a))\prec \mathbb{U}$.

\item For $j<\mu$ and $\eta\in (I^{\U{\alpha_j}})_<$, if $A\subseteq F_{j,\eta}=\dcl (\{a_{i,k,\nu}: (i,k)\in U,\, i<j,\, \nu\in (\Rg(\eta)^{\U{\alpha_i}})_<\})$ with $|A|<\kappa$ and non-algebraic $p\in S(A)$ then there exists $k<\mu$ with $(j,k)\in U$ such that $a_{j,k,\eta}\models p$. Moreover, if $p$ is finitely satisfiable in $A$ then so is $\tp(a_{j,k,\eta}/F_{j,\eta})$. 

\item  If in addition, $(I,<)$ is well-ordered and $\alpha$ satisfies $\alpha_i=(i\mod \kappa)$ then 
\begin{enumerate}
\item for any $A\subseteq \dcl(\Rg(a))$ with $|A|<\kappa$ there exist $i<\mu$ and $\eta\in (I^{\U{\alpha_i}})_<$ satisfying $A\subseteq F_{i,\eta}$;
\item $\dcl(\Rg(a))$ is $\kappa$-saturated.
\item Assume that $(I,<)$ is a cardinal with $\cf(I)\geq \kappa$. For any infinite $A\subseteq B \subseteq \dcl(\Rg(a))$, with $|B|<\kappa$, there is a non principal ultrafilter $\mathcal{D}$ on $A$ such that $(p_{\mathcal{D}})^{\otimes I}|B$ is realized in $\dcl(\Rg(a))$. 

\end{enumerate}
\end{enumerate}

%

%
\end{proposition}
\begin{proof}
Since $\mathbb{U}$ has definable Skolem functions, for any $A\subseteq \mathbb{U}$, $\dcl(A)$ is an elementary substructure of $\mathbb{U}$.

Let $j<\mu$ and assume we found $\{U_{i}\subseteq \mu:i<j\}$ and $a_{<j}=\langle a_{i,k,\eta}:i<j,\, k\in U_i,\, \eta\in (I^{\U{\alpha_i}})_<\rangle $ such that $a_{<j}$ is $((\alpha^\prime)^{<j}, I)$-indiscernible, where $((\alpha^\prime)^{<j})_{i,k}=\alpha_i$ for $i<j$ and $k\in U_i$.

If $(I^{\U{\alpha_j}})_<=\emptyset$ then there is nothing to do. Otherwise, fix some $\eta^*\in (I^{\U{\alpha_j}})_<$ and a $\dcl$-closed subset 
\[A_{j,\eta^*}\subseteq C_j=\dcl(\{a_{i,k,\nu}:i<j,\, k\in U_i,\, \nu\in (I^{\U{\alpha_i}})_<\})\] of cardinality at most $\mu$.


For any $A\subseteq A_{j,\eta^*}$ with $|A|<\kappa$ and a non-algebraic type $p\in S_1(A)$ we choose an extension $\tilde p$ of $p$ to $A_{j,\eta^*}$ and a non-principal ultrafilter $\mathcal{D}(p)$ on $A_{j,\eta^*}$, such that $\tilde p = p_{\mathcal{D}(p)}|A_{j,\eta^*}$, in the following way
\begin{list}{•}{}
\item if $p$ is finitely satisfiable in $A$ then let $\mathcal{D}(p)$  be any (necessarily unique) ultrafilter (on $A_{j,\eta^*}$) extending the filter $\{B\subseteq A_{j,\eta^*}:\varphi(A,a)\subseteq B,\, \varphi(x,a)\in p\}$. We let $\tilde p=p_{\mathcal{D}(p)}|A_{j,\eta^*}$;
\item otherwise, let $\tilde p$ be any non-algebraic extension of $p$ to $A_{j,\eta^*}$. Since $A_{j,\eta^*}$ is a model, $\tilde p$ is finitely satisfiable in $A_{j,\eta^*}$. Let $\mathcal{D}(p)$ be any non-principal ultrafilter extending $\{\varphi(A_{j,\eta^*},b):\varphi(x,b)\in \tilde p\}$.
\end{list}

We note that there are at most $\mu^{<\kappa}\leq\mu$ subsets $A\subseteq A_{j,\eta^*}$ with $|A|<\kappa$ and for all such $A$ there are at most $2^{\kappa+|\mathcal{L}|}\leq \mu$ types on $A$.

Let $U_j\subseteq \mu$ be such that $\langle (p_{j,k,\eta^*},\mathcal{D}_{j,k,\eta^*}):k\in U_j\rangle$ enumerates the set of pairs $(\tilde p,\mathcal{D}(p))$ for non-algebraic $p\in S_1(A)$.
%

By the induction hypothesis, any partial order-isomorphism $\pi$ of $I$ induces a partial elementary map $\widehat \pi$ whose domain is 
\[\dcl\left( \{a_{i,k,\nu}:i<j,\, k\in U_i,\, \Rg(\nu)\subseteq \dom(\pi)\}\right),\]
mapping $a_{i,k,\nu}\mapsto a_{i,k,\pi(\nu)}$, where $\pi(\nu)=\pi\circ \nu$. Note that for any $\pi_1,\pi_2$, if $\pi_1\circ \pi_2$ makes sense then $\widehat{\pi_1\circ \pi_2}=\widehat{\pi_1}\circ \widehat{\pi_2}$.

As a result, for any $\rho\in (I^{\U{\alpha_j}})_<$ the unique order-isomorphism $\pi_{\eta^*,\rho}:\eta^*\to \rho$ induces a partial elementary map, $\widehat{\pi_{\eta^*,\rho}}: A_{j,\eta^*}\to A_{j,\rho}$, where $A_{j,\rho}=\widehat\pi(A_{j,\eta^*})$, which is given by $a_{i,k,\nu}\mapsto a_{i,k,\pi_{\eta^*,\rho}(\nu)}$. For every $k\in U_j$ let $p_{j,k,\rho}=\widehat{\pi_{\eta^*,\rho}}(p_{j,k,\eta^*})\in S(A_{j,\rho})$ and let $\mathcal{D}_{j,k,\rho}=\widehat{\pi_{\eta^*,\rho}}(\mathcal{D}_{j,k,\eta^*})$. 
\begin{claim}
For any $\rho\in (I^{\U{\alpha_j}})_<$ and $\pi$ a partial isomorphism of $I$ whose domain contains  $\Rg(\rho)$, $\widehat{\pi}(A_{j,\rho})=A_{j,\pi(\rho)}$, $\widehat{\pi}(p_{j,k,\rho})=p_{j,k,\pi(\rho)}$ and $\widehat{\pi}(\mathcal{D}_{j,k,\rho})=\mathcal{D}_{j,k,\pi(\rho)}$.
\end{claim}
\begin{claimproof}
There is no harm in restricting $\pi$ to $\Rg(\rho)$.  Let $\pi_{\eta^*,\rho}:\eta^*\to \rho$ be the unique order-isomorphism, so $\pi\circ\pi_{\eta^*,\rho}$ is the unique isomorphism from $\eta^*$ to $\pi(\rho)$ and thus equal to $\pi_{\eta^*,\pi(\rho)}$. Hence $\pi=\pi_{\eta^*,\pi(\rho)}\circ \pi_{\eta^*,\rho}^{-1}$. The result follows.
\end{claimproof}

Let $((I^{\U{\alpha_j}})_<,<^{\text{lex}})$ be the lexicographic ordering and let $(U_j,<)$ be the order induced from $\mu$. By induction on $k\in U_j$, by compactness we may find a sequence $\langle a_{j,k,\eta}: \eta\in (I^{\U{\alpha_j}})_<\rangle$ satisfying that for any $\eta\in (I^{\U{\alpha_j}})_<$
\[a_{j,k,\eta}\models p_{\mathcal{D}_{j,k,\eta}}|B_{j,k,\eta},\]
where
\[B_{j,k,\eta}=C_j\cup \{a_{j,l,\rho}: l<k,\, l\in U_j,\rho\in (I^{\U{\alpha_j}})_<\}\cup \{a_{j,k,\rho}:\rho<^{\text{lex}}\eta\}\]

We show $((\alpha^\prime)^{\leq j},I)$-indiscernibility by induction on $\{(i,l):i\leq j,\, l\in U_i\}$ (with the lexicographic ordering). In other words, we assume that for any $\langle (i_r,l_r): r<n\rangle$, with $(i_r,l_r)<(j,k)$ and $l_r\in U_r$, $\langle \eta_r\in (I^{\U{\alpha_{i_r}}})_<:r<n\rangle$ and a partial isomorphism $\pi$ of $I$, whose domain contains $\bigcup \{\Rg(\eta_r):r<n\}$,
\[\tp(a_{i_0,l_0,\eta_0},\dots,a_{i_{n-1},l_{n-1},\eta_{n-1}})=\tp(a_{i_0,l_0,\pi(\eta_0)},\dots,a_{i_{n-1},l_{n-1},\pi(\eta_{n-1})}).\]
We wish to show the same statement for $(i_r,l_r)\leq (j,k)$.

We prove by induction on $n$ that for any $\bar b\subseteq \{a_{i,l,\eta}: (i,l)<(j,k),\,\eta\in (I^{\U{\alpha_i}})_<,\, l\in U_{i}\}$, any $\eta_{n-1}<^{\text{lex}}\dots<^{\text{lex}}\eta_0\in (I^{\U{\alpha_j}})_<$ and any partial isomorphism $\pi$ of $(I,<)$ whose domain contains 
\[\Rg(\eta_0)\cup\dots\cup\Rg(\eta_{n-1})\cup\bigcup\{\Rg(\eta):a_{i,l,\eta}\in \bar b\},\]
\[\tp(a_{j,k,\eta_0},\dots,a_{j,k,\eta_{n-1}}, \bar b)=\tp(a_{j,k,\pi(\eta_0)},\dots,a_{j,k,\pi(\eta_{n-1})},\widehat{\pi}(\bar b)).\]

Let $\varphi(x_0,\dots,x_{n-1},\bar b)$ be some formula, where $\bar b$ is as above.
We show that \[\varphi(x_0,\dots,x_{n-1},\bar b)\in p_{\mathcal{D}_{j,k,\eta_0}}\otimes\dots\otimes p_{\mathcal{D}_{j,k,\eta_{n-1}}}\iff\]\[\varphi(x_0,\dots,x_{n-1},\pi(\bar b))\in p_{\mathcal{D}_{j,k,\pi(\eta_0)}}\otimes\dots\otimes p_{\mathcal{D}_{j,k,\pi(\eta_{n-1})}}.\]
Indeed, if $\varphi(x_0,\dots,x_{n-1},\bar b)\in p_{\mathcal{D}_{j,k,\eta_0}}\otimes\dots\otimes p_{\mathcal{D}_{j,k,\eta_{n-1}}}$ then by the choice of the $a_{j,k,\eta}$'s,  $\varphi(a_{j,k,\eta_0},\dots,a_{j,k,\eta_{n-1}},\bar b)$ holds and thus $X=\varphi(A_{j,\eta_0},a_{j,k,\eta_1},\dots,a_{j,k,\eta_{n-1}},\bar b)\in \mathcal{D}_{j,k,\eta_0}$. By the Claim, $\widehat\pi(X)\in \mathcal{D}_{j,k,\pi(\eta_0)}$. By the induction hypothesis (on $n$), $\widehat\pi$ is elementary on $a_{j,k,\eta_1}\cup \dots\cup a_{j,k,\eta_{n-1}}\cup \bar b$ and as a result, \[\widehat\pi(X)=\varphi(A_{j,\pi(\eta_0)},a_{j,k,\pi(\eta_1)},\dots,a_{j,k,\pi(\eta_{n-1})},\widehat \pi(\bar b))\in \mathcal{D}_{j,k,\pi(\eta_0)},\]
 and as $\pi$ preserves $<^{\text{lex}}$,
\[\varphi(x_0,a_{j,k,\pi(\eta_1)},\dots,a_{j,k,\pi(\eta_{n-1})},\widehat \pi(\bar b))\in p_{\mathcal{D}_{j,k,\pi(\eta_0)}}|B_{j,k,\pi(\eta_0)}.\] 
As a result, by the choice of the $a_{j,k,\eta}$'s,
\[\varphi(a_{j,k,\pi(\eta_0)},a_{j,k,\pi(\eta_1)},\dots,a_{j,k,\pi(\eta_{n-1})},\widehat \pi(\bar b)) \text{ holds,}\] and thus 
 \[\varphi(x_0,\dots,x_{n-1},\widehat\pi(\bar b))\in p_{\mathcal{D}_{j,k,\pi(\eta_0)}}\otimes\dots\otimes p_{\mathcal{D}_{j,k,\pi(\eta_{n-1})}}.\]
This proves $(1)$, i.e. $a=\langle a_{i,k,\eta}: i<\mu,\, k\in U_i,\, \eta\in (I^{\U{\alpha_i}})_<\rangle$ is $(\alpha^\prime,I)$-indiscernible.

To prove $(2)$ and $(3)$ recalling the beginning of the proof of $(1)$, let \[A_{j,\eta}=F_{j,\eta}=\dcl (\{a_{i,k,\nu}: i<j, k\in U_{i}, \nu\in (\Rg(\eta)^{\U{\alpha_i}})_<\}).\]
$(2)$ follows immediately once we observe the following:
\begin{list}{•}{}
\item $|A_{j,\eta}|\leq \mu$. This follows from the following inequalities \[\mu\cdot {|\alpha_j|^{<\kappa}}\leq \mu\cdot \kappa^{<\kappa}\leq \mu.\] 
\item For any order-preserving partial isomorphism $\pi$ of $I$, whose domain contains $\Rg(\eta)$, by the induction hypothesis on $j$, $\widehat\pi(A_{j,\eta})=A_{j,\pi(\eta)}.$ 
\end{list}

Now assume that $(I,<)$ is well ordered, that $\alpha$ satisfies $\alpha_i=(i\mod \kappa)$ and let $A\subseteq a$ with $|A|<\kappa$. Since $(I,<)$ is well-ordered there exist an ordinal $\beta$ and an order isomorphism, $\eta:\beta\to \bigcup_{a_{j,l,\nu}\in A} \Rg(\nu)$. Since $\kappa$ is a regular cardinal and for every $a_{j,l,\nu}\in A$ we have $|\Rg(\nu)|=|\alpha_j|<\kappa$, it follows that $\beta<\kappa$.

Since $\kappa<2^\kappa\leq\mu$ and $\mu^{<\kappa}=\mu$ (so $\cf(\mu)\geq \kappa$), $\widehat j=\sup_{a_{j,l,\nu}\in A} j<\mu$. Let $i=\widehat j\cdot \kappa+\beta<\mu$. By the choice of $\alpha$, $\alpha_i=\beta$ and $\eta\in (I^{\U{\alpha_i}})_<$. This implies that $A\subseteq F_{i,\eta}$. This gives $(3.a)$. Now $(3.b)$ follows by $(2)$. 

Item $(3.c)$ follows from the construction, we elaborate. Let $A\subseteq B$ with $|B|<\kappa$. By $(3.a)$ there exists $j<\mu$ and $\eta\in (I^{\U{\alpha_j}})_<$ such that $B\subseteq F_{j,\eta}$. Let $p\in S^{f.s.}(A)$ a non-algebraic type which is finitely satisfiable in $A$. As $\cf(I)\geq\kappa$, there is some $\Rg(\eta)<\gamma^*\in I$ and by the Claim above for any $\Rg(\eta)<\gamma\in I$,
\[F_{j+1,\eta^\frown\gamma}=\pi_{\eta^\frown \gamma^*,\eta^\frown \gamma}(F_{j+1,\eta^\frown\gamma^*}).\]
Let $\mathcal{D}$ be the non-principle ultrafilter on $F_{j+1,\eta^\frown\gamma^*}$ corresponding to $p$, as chosen above. Hence $p_{\mathcal{D}}$ is finitely satisfiable in $A$.
Observe that since $A\subseteq F_{j,\eta}$, $\pi_{\eta^\frown\gamma^*,\eta^\frown\gamma}$ fixes $A$ point-wise. It follows by Remark \ref{R:global-ultrafilter} that for every $\Rg(\eta)<\gamma\in I$, $p_{\mathcal{D}}=p_{\mathcal{D}_{j+1,k,\eta^\frown\gamma}}$. 

Let $k\in U_j$ be such that $\tp(a_{j+1,k,\eta^\frown\gamma^*}/F_{j+1,\eta^\frown\gamma^*})=p_{\mathcal{D}}|F_{j+1,\eta^\frown\gamma^*}$.  By the choice of elements, for any $\Rg(\eta)<\gamma\in I$,
\[a_{j+1,k,\eta^\frown\gamma}\models p_{\mathcal{D}}|F_{j,\eta}\langle a_{j+1,k,\eta^\frown\delta}: \Rg(\eta)<\delta<\gamma\rangle.\] We end by noting that since we are assuming that $(I,<)$ is a cardinal and $\cf(I)\geq \kappa>\alpha_j$, $|\{\gamma:\Rg(\eta)<\gamma\in I\}|=|I|$.
%
%
\end{proof}

In stable theories, for any infinite indiscernible sequence $I$ over some set $A$ one may take the limit type defined by 
\[\lim(I)=\{\varphi(x,c): \text{$\varphi(a,c)$ holds for cofinitely many $a\in I$}\}.\]
It is a consistent complete type by stability. It is obviously finitely satisfiable in $I$. Moreover, if $\mathcal{D}$ is a non-principal ultrafilter on $I$, then $p_\mathcal{D}=\lim(I)$. We often write $\lim(I/A)=\lim(I)|A$.

The following is \cite[Lemma III.3.10]{classification}, we give a proof for completeness.
\begin{lemma}\label{L:saturation}
Let $T$ be a stable theory and $M\models T$. If $M$ is $(\kappa(T)+\aleph_1)$-saturated and every countable indiscernible sequence over $A\subseteq M$, with $|A|<\kappa(T)$, in $M$ can be extended to one of cardinality $\lambda$ then $M$ is $\lambda$-saturated.
\end{lemma}
\begin{proof}
We may assume that $\lambda>\kappa(T)+\aleph_1$.
By passing to $M^{eq}$ (and $T^{eq}$) there is no harm in assuming that $T$ eliminates imaginaries. 
Let $p\in S(C)$ with $C\subseteq  M$ and $|C|<\lambda$. Let $B\subseteq C$ with $|B|<\kappa(T)$ such that $p$ does not fork over $B$. Let $q\supseteq p$ be its non-forking global extension. Since $M$ is $\kappa(T)$-saturated, we may find a sequence of elements $S=\langle b_i :i<\omega \rangle\subseteq M$ satisfying $b_i\models q|B\langle b_j:j<i\rangle$. Note that $q|BS$ is stationary by \cite[Corollary III.2.11]{classification}. 

Since $M$ is $(\kappa(T)+\aleph_1)$-saturated, we may find a Morley sequence $I=\langle a_i:i<\omega\rangle$ of $q$ over $SB$, i.e. $a_i\models q|SBa_{<i}$ and $a_i\in M$. It follows that $I$ is also a Morley sequence of $q$ over $\acl(B)$\footnote{It is standard to see that $I$ is independent and indiscernible over $\acl(B)$. On the other hand, since $q|BS$ is stationary, it isolates a complete type over $\acl(BS)$.}. Let $I\subseteq J\subseteq M$ be an indiscernible sequence (over $B$) of cardinality $\lambda$. As a result, $J$ is also a Morley sequence of $q$ over $\acl(B)$.

By \cite[Lemma III.1.10(2)]{classification}, $\lim(J/M)=q|M$ and in particular $\lim(J/C)=p$. By \cite[Corollary III.3.5(1)]{classification}, there is $J_0\subseteq J$ with $J\setminus J_0$ indiscernible over $C$ and $|J_0|\leq \kappa(T)+|C|<\lambda$. In particular, $|J\setminus J_0|\geq \aleph_0$ and thus for every $c\in J\setminus J_0$, $p=\tp(c/C)$.
%
\end{proof}

\begin{definition}
Let $T$ be a theory. We say that $M\models T$ is an \emph{infinitary EM-model based on $(\alpha,I)$} if $M=\dcl(a)$, where $a$ is an $(\alpha,I)$-indiscernible sequence for $(\alpha,I)$ as in Definition \ref{D:alpha-I}.
\end{definition}

\begin{lemma}\label{L:underlying set of Inf-EM is Inf-EM}
Let $T$ be a any theory.
Let $\kappa\geq \aleph_0$ be a cardinal, $(I,<)$ a linearly ordered set, $\alpha\in \kappa^U$, where $U$ is a set. If $a=\langle a_{i,\eta}:i\in U,\, \eta\in (I^{\U{\alpha_i}})_<\rangle$ is an $(\alpha,I)$ indiscernible sequence, in some model $M\models T$, then there exists some set $\widehat U$, with $|\widehat{U}|\leq  |T|\cdot |U|\cdot \kappa^{<\kappa}$, and $\widehat{\alpha}\in \kappa^{\widehat{U}}$ and an $(\widehat{\alpha},I)$-indiscernible sequence $b$ whose underlying set is $\dcl(a)$.
\end{lemma}
\begin{proof}
%
%
For any $p\subseteq \kappa$ let $\varphi_p:\otp(p)\to p$ be the unique order isomorphism. Let $\mathcal{F}$ be the collection of all $\emptyset$-definable functions. We consider the family $\widehat U$ of tuples  $(f(\bar v), s_0,p_0,\dots,s_{|\bar v|-1},p_{|\bar v|-1})$ satisfying
\begin{list}{•}{}
\item $f(\bar v)\in \mathcal{F}$, 
\item $s_0,\dots, s_{|\bar v|-1}\in U$,
\item for any $i<|\bar v|$, $p_i\subseteq \kappa$ with $\otp(p_i)=\alpha_{s_i}$,
\item $\bigcup_{i<|\bar v|}p_i\in \text{Ord}$.
\end{list}
We note that $|\widehat U|\leq |T|\cdot |U|^{<\aleph_0}\cdot (\kappa^{<\kappa})^{<\aleph_0}\leq |T|\cdot |U|\cdot \kappa^{<\kappa}.$
Let $\widehat{\alpha}\in \kappa^{\widehat U}$ be the function mapping $x=(f(\bar v), s_0,p_0,\dots,s_{|\bar v|-1},p_{|\bar v|-1})$ to $\bigcup_{i<|\bar v|} p_i<\kappa$.  For any $x=(f(\bar v), s_0,p_0,\dots,s_{|\bar v|-1},p_{|\bar v|-1})\in\widehat{U}$ and $\eta\in (I^{\U{\widehat{\alpha}_x}})_<$ set \[b_{x,\eta}=f(a_{s_0,\eta\restriction p_0\circ \varphi_{p_0}},\dots,a_{s_{|\bar v|-1},\eta\restriction p_{|\bar v|-1}\circ \varphi_{p_{|\bar v|-1}}}).\] Note that for any $j<|\bar v|$, $(\eta\restriction p_j)\circ \varphi_{p_j}\in (I^{\U{\alpha_{s_j}}})_<$.

Let $b=\langle b_{x,\eta}: x\in \widehat{U},\, \eta\in (I^{\U{\widehat{\alpha}_x}})_<\rangle$. We  will show that $b$ is $(\widehat \alpha,I)$-indiscernible.

Let $\langle x_j\in \widehat{U}: j<k\rangle$, $\langle \eta_j\in (I^{\U{\widehat{\alpha}_{x_j}}})_<:j<k\rangle$ and $\pi$ be a partial isomorphism of $(I,<)$ whose domain contains $\bigcup_{j<k}\Rg(\eta_j)$. For $j<k$, we write $b_{x_j,\eta_j}=f_j(a_{s_{j,0},\eta_j\restriction p_{j,0}\circ \varphi_{p_{j,0}}},\dots,a_{s_{j,|\bar v_j|-1},\eta_j\restriction p_{j,|\bar v_j|-1}\circ \varphi_{p_{j,|\bar v_j|-1}}})$.

Since $a$ is $(\alpha,U)$-indiscernible, the type of 
\[\langle a_{s_{j,0},\eta_j\restriction p_{j,0}\circ \varphi_{p_{j,0}}},\dots,a_{s_{j,|\bar v|-1},\eta_j\restriction p_{j,|\bar v_j|-1}\circ \varphi_{p_{j,|\bar v_j|-1}}}:j<k\rangle\] is equal to the type of
\[\langle a_{s_{j,0},\pi(\eta_j\restriction p_{j,0}\circ \varphi_{p_{j,0}})},\dots,a_{s_{j,\bar v|-1},\pi(\eta_j\restriction p_{j,|\bar v_j|-1}\circ \varphi_{p_{j,|\bar v_j|-1}})}:j<k\rangle,\]
and consequently the type of $\langle b_{x_j,\eta_j}:j<k\rangle$ is the equal to the type of $\langle b_{x_j,\pi(\eta_j)}:j<k\rangle$.

Finally, let $c\in \dcl(a)$. I.e. there is a definable function $f(\bar v)\in \mathcal{F}$, and $a_{i_0,\eta_0},\dots a_{i_{|\bar v|-1},\eta_{|\bar v|-1}}\in a$ such that $c= f(a_{i_0,\eta_0},\dots a_{i_{|\bar v|-1},\eta_{|\bar v|-1}})$. Let $r=\bigcup_{i<|\bar v|}\Rg(\eta_{i})$ and $\psi:r\to \otp(r)$ be the unique order isomorphism. For any $j<|\bar v|$ set $p_j=\psi(\Rg(\eta_j))$. Now note that for $x=(f(\bar v),i_0,p_0,\dots,i_{|\bar v|-1},p_{|\bar v|-1})$. So for $\eta=\psi^{-1}$, $c=b_{x,\eta}$ (because e.g. $\eta\restriction p_0\circ \varphi_{p_0}=\psi^{-1}\restriction \psi(\Rg(\eta_0))\circ \varphi_{\psi(\Rg(\eta_0))}=\psi^{-1}\restriction \psi(\Rg(\eta_0))\circ \psi\circ \eta_0=\eta_0$).
\end{proof}

\begin{theorem}\label{T:existence of gen em model in stable}
The following are equivalent for a complete $\mathcal{L}$-theory $T$:
\begin{enumerate}
\item $T$ is stable.

\item Let $\kappa,\mu$ and $\lambda$ be cardinals satisfying $\kappa=\cf(\kappa)\geq \kappa(T)+\aleph_1$, $\mu^{<\kappa}=\mu\geq 2^{\kappa+|T|}$ and $\lambda=\lambda^{<\kappa}\geq \mu$ and let $T\subseteq T^{sk}$ be an expansion with definable Skolem functions such that $|T|=|T^{sk}|$ in a language $\mathcal{L}\subseteq \mathcal{L}^{sk}$. Then there exists an infinitary EM-model  $M^{sk}\models T^{sk}$ based on $(\alpha,\lambda)$, where $\alpha\in \kappa^U$ for some set $U$ of cardinality at most $\mu$, such that $M=M^{sk}\restriction \mathcal{L}$ is saturated of cardinality $\lambda$.

\item Let $\kappa,\mu$ and $\lambda$ be cardinals satisfying $\kappa=\cf(\kappa)\geq \kappa(T)+\aleph_1$, $\mu^{<\kappa}=\mu\geq 2^{\kappa+|T|}$ and $\lambda=\lambda^{<\kappa}\geq \mu$. Then there exists a saturated model of cardinality $\lambda$.

\end{enumerate}

\end{theorem}
\begin{remark}
For example, the assumptions in (2) hold for $\lambda=\mu=2^{\kappa+|T|}$ for any $\kappa=\cf(\kappa)\geq \kappa(T)+\aleph_1$. 
\end{remark}
\begin{proof}
$(1)\implies (2)$. 
In the following, the superscript $^{sk}$ means that we work in $T^{sk}$.
 
We apply Proposition \ref{P:existence-of-indisc}(1,3) with $\mathbb{U}$ there being a monster model for $T^{sk}$ and $(I,<)=(\lambda,<)$.  Consequently, there exists an $(\alpha^\prime,\lambda)$-indiscernible sequence $a$, where $\alpha^\prime$ is as in the proposition.
Let $M^{sk}=\dcl^{sk}(a)$ and $M=M^{sk}\restriction \mathcal{L}$. Note that $|M^{sk}|=|M|=\mu\cdot \lambda^{<\kappa}=\lambda$.

Towards applying Lemma \ref{L:saturation}, note that $M$ is indeed $(\kappa(T)+\aleph_1)$-saturated by Proposition \ref{P:existence-of-indisc}(3.b) and the assumption on $\kappa$.  Let $I\subseteq M$ be an infinite countable indiscernible sequence over some $B\subseteq M$ with $|B|<\kappa(T)\leq \kappa$.

%
Since $\lambda<\lambda^{\cf(\lambda)}$, necessarily $\cf(\lambda)\geq \kappa$ so by Proposition \ref{P:existence-of-indisc}(3.c) there is a  non principal ultrafilter $\mathcal{D}$ on $I$ and elements $\langle a_i\in \dcl(\Rg(a)):i<\lambda\rangle$ satisfying that 
\[a_i\models p^{sk}_{\mathcal{D}}|BI\langle a_k: k<i\rangle\] for any $i<\lambda$.
Let $p_{\mathcal{D}}$ be the restriction of $p^{sk}_{\mathcal{D}}$ to $\mathcal{L}$. Thus $p_{\mathcal{D}}=\lim(I)$ and for every $i<\lambda$
\[a_i\models \lim(I)|BI\langle a_k: k<i\rangle.\]
By stability, $I+\langle a_i:i<\lambda\rangle$ is indiscernible over $B$ (see also \cite[Exercise 2.25]{guidetonip} and \cite[Lemma III.1.7(2)]{classification}). By Lemma \ref{L:saturation}, $M$ is saturated.

$(2)\implies (3)$ is obvious.

$(3)\implies (1)$. Let $\kappa$ be any cardinal satisfying $\kappa=\cf(\kappa)\geq \kappa(T)+\aleph_1$ and let $\lambda = \mu = \beth_\kappa(\kappa)$. Then $\lambda^{<\kappa} = \lambda$ because $\kappa$ is regular. Indeed, any function from some $\xi<\kappa$ to $\lambda$ is a function to $\beth_\alpha(\kappa)$ for some $\alpha<\kappa$. So $\lambda^\xi = \sup_{\alpha<\kappa} (\beth_\alpha(\kappa)^\xi)$. But $\sup_{\alpha<\kappa}(\beth_\alpha(\kappa)^\xi)= \sup_{\alpha<\kappa}(\beth_{\alpha+1}(\kappa)^\xi)$, and $(\beth_{\alpha+1}(\kappa))^\xi=(2^{\beth_\alpha(\kappa)})^\xi = 2^{\beth_\alpha(\kappa)\cdot \xi} = 2^{\beth_\alpha(\kappa)}$ because $\kappa>\xi$. Consequently, $\lambda^\xi = \lambda$ and $\lambda^{<\kappa} = \lambda$.

Hence, by (3), there is a saturated model of size $\lambda$. On the other hand, since $\lambda$ is singular (of cofinality $\kappa<\lambda$), $\lambda^{<\lambda}>\lambda$ and as a result by \cite[Theorem VIII.4.7]{classification},  $T$ is $\lambda$-stable (and hence stable).
 
\end{proof}

\section{Order-Type graphs with large chromatic number} \label{S:order type graphs}
In this section we discuss graphs whose vertices are (possibly infinite) increasing sequences, where the edge relation is determined by the order type. More specifically, our main interest in this section is the following type of graphs.
\begin{definition}
Let $(I,<)$ and $(J,<)$ be linearly ordered sets and $\bar a\neq \bar b\in (I^{\U J})_<$ be increasing sequences. We define a graph $E^J_{\bar a,\bar b}$ and a directed graph $D_{\bar a,\bar b}^J$ on $(I^{\U J})_<$ by:
\begin{itemize}
\item $\bar c \E^J_{\bar a,\bar b} \bar d \iff \otp(\bar c,\bar d)=\otp(\bar a,\bar b)\vee \otp(\bar d,\bar c)=\otp(\bar a,\bar b)$
\item $\bar c \D^J_{\bar a,\bar b} \bar d \iff \otp(\bar c,\bar d)=\otp(\bar a,\bar b).$
\end{itemize}
We omit $J$ from $E^J_{\bar a,\bar b}$ and $D^J_{\bar a,\bar b}$ when it is clear from the context.

We call these graphs the \emph{(directed) order-type graphs}.
\end{definition}

\begin{remark}
Although it will not define a graph, we sometimes use the notation $D_{\bar a,\bar b}$ and $E_{\bar a,\bar b}$ even if $\bar a=\bar b$.
\end{remark}

In Section \ref{ss:embed-shift} we isolate a family of order-type graphs whose members contain all finite graphs of $\Sh_m(\omega)$ for a certain integer $m$ (Corollary \ref{C:k-ord-cov-implies shift}). In Section \ref{ss:orde-type-large-chrom} we show that order-type graphs with large chromatic number fall into this family (Theorem \ref{T:shift graphs in order-type-graphs}).

\subsection{Embedding shift graphs into order-type graphs}\label{ss:embed-shift}
\begin{definition}
Let $(I,<)$ and $(J,<)$ be linearly ordered sets, $\bar a,\bar b\in (I^{\U J})_<$ be increasing sequences and $0<k<\omega$.

We say that $\langle \bar a,\bar b\rangle$ is \emph{k-orderly} if 
there exists a finite partition $Conv(\Img(\bar a)\cup\Img(\bar b))=C_0\cup\dots \cup C_k$ by convex increasing subsets satisfying 
that for every $n<k$ and $i\in J$, $a_i\in C_n\iff b_i\in C_{n+1}$;
\end{definition}

Recall the following from \cite{1196}.
\begin{definition}
For any linearly ordered set $(A,<)$ and $k\geq 1$, let $\LSh_k(A)$ be the directed graph $((A^{\U k})_<,D)$, were $(\eta,\rho)\in D$ if and only if $\eta(i)=\rho(i-1)$ for $0<i<k$ (if $k>1$) and $\eta(0)<\rho(0)$ (if $k=1$).
\end{definition}

\begin{lemma}\label{L:k-orderly}
Let $0<k<\omega$ be an integer, $\alpha,\delta$ be ordinals, $(I,<)$ any infinite linearly ordered set satisfying $(\delta\times (2\cdot \alpha +1)^k,<_{lex})\subseteq (I,<)$. Let $\bar a,\bar b\in (I^{\U{\alpha}})_<$.
If $\langle \bar a,\bar b\rangle $ is $k$-orderly then there exists a function $\varphi:\LSh_k(\delta)\to (I^{\U{\alpha}})_<$, satisfying that for any $\eta,\rho\in \LSh_k(\delta)$, if $(\eta,\rho)\in D$ then $\varphi(\eta)\D_{\bar a,\bar b} \varphi(\rho)$.
\end{lemma}
\begin{proof}
Assume that $Conv(\Img(\bar a),\Img(\bar b))=C_0\cup\dots\cup C_k$, as in the definition.

Let $\alpha^*=\alpha\cup\{\beta^-:\beta<\alpha\}\cup\{\infty\}$, where the $\beta^-$'s are immediate predecessors and $\infty$ is a maximal element, i.e. for any $\beta< \gamma< \alpha$
\begin{list}{•}{}
\item $\gamma<\beta^-<\beta$,
\item  $\gamma^-<\beta^-$ if and only $\gamma<\beta$ and 
\item $\gamma<\infty$.
\end{list}
For any $S\subseteq \alpha$, let $S^*$ be $S\cup\{s^-:s\in S\}\cup\{0^-,\infty\}$.

For any $x\in \left(\alpha^*\right)^n$, we denote $x^-$ the immediate predecessor of $x$ in the lexicographic order if it exists, and otherwise let $x^-=x$. Note that for any $x=(x_0,\dots,x_{n-1})\neq(0^-,\dots,0^-)\in \left(\alpha^*\right)^n$, if the maximal $l<n$ with $x_l\neq 0^-$ satisfies $x_l< \alpha$ then $x$ has an immediate predecessor.

Note that the order type of $\alpha^*$ is $2\cdot \alpha +1$, so by the assumption on $I$ we may replace $I$ by an isomorphic copy to get that $(\delta\times (\alpha^*)^k,<_{lex})\subseteq (I,<)$.

For any $0\leq i\leq k-1$ let $S_i=\{\beta< \alpha: a_\beta\in C_i\}$ and 
let $\mathcal{G}=\{\bar g=\langle g_i: S_i\cup\{\infty\} \to (S_{k-1}^*\times\dots\times S_0^*,<_{lex}): i<k\rangle: \text{ $g_i$ increasing}\}$.

%
For any $\bar g \in\mathcal{G}$ and $\eta\in (\delta^{\U k})_<$ let $f_{\eta,\bar g}\in (I^{\U \alpha})_<$ be defined by
\[f_{\eta,\bar g}(\beta)=(\eta(n_\beta),g_{n_\beta}(\beta))\in \delta\times (S^*_{k-1}\times\dots\times S_0^*)\subseteq I,\] where $\beta\in S_{n_\beta}$. We note that $f_{\eta,\bar g}$ is increasing: if $n_{\beta_1}<n_{\beta_2}$ then $\eta(n_{\beta_1})<\eta(n_{\beta_2})$. If $n_{\beta_1}=n_{\beta_2}$ then the results follows since $g_{n_{\beta_1}}=g_{n_{\beta_2}}$ is increasing.

\begin{claim}
There exists $\bar g\in \mathcal{G}$ such that for any $\eta,\rho\in \Sh_k(\delta)$ satisfying $\eta(i)=\rho(i-1)$ for $0<i<k$ (if $k>1$) or $\eta(0)<\rho(0)$ (if $k=1$), $f_{\eta,\bar g}\D_{\bar a,\bar b} f_{\rho,\bar g}$.
\end{claim}
\begin{claimproof}
For the purpose of this proof, for $1\leq i\leq k$ let $\pi_i:S^*_{k-1}\times\dots\times S_0^*\to S^*_{k-1}\times\dots\times S^*_{k-i}$ be the projection on the first $i$ coordinates. 
We choose increasing functions $g_i:S_i\cup\{\infty\} \to S_{k-1}^*\times\dots S^*_{k-i}\times \{0^-\}\times\dots\times \{0^-\}$ by downwards induction on $i<k$. Define $g_{k-1}$ by setting $g_{k-1}(\beta)=(\beta,0^-,\dots,0^-)$ for $\beta\in S_{k-1}\cup\{\infty\}$.


 Assume that $g_i$ was defined and we want to define $g_{i-1}$. 

For any $\beta\in S_{i-1}$ if there is $\gamma\in S_i$ minimal such that $b_\beta\leq a_\gamma$ then define
\[g_{i-1}(\beta)=
\begin{cases}
g_i(\gamma)=(\pi_{k-i}(g_i(\gamma)),0^-,0^-,\dots, 0^-) & \text{ if $a_\gamma=b_\beta$} \\
(\pi_{i}(g_i(\gamma))^-,\beta,0^-,\dots, 0^-) & \text{ otherwise.}
\end{cases}
\]
If such a minimal $\gamma\in S_i$ does not exists then we define
\[g_{i-1}(\beta)=(\pi_{i}(g_i(\infty)),\beta,0^-,\dots,0^-).\]
Lastly, \[g_{i-1}(\infty)=(\pi_{i}(g_i(\infty)),\infty,0^-,\dots,0^-).\]

\begin{subclaim}
For any $0\leq i\leq k-1$, and for every $\beta\in S_i$, $\pi_{i}(g_i(\beta))$ has an immediate predecessor. I.e.,  for every $\gamma<\beta\in S_i$, $\pi_{i}(g_i(\beta))>\pi_{i}(g_i(\beta))^-$.

For any $0\leq i\leq k-1$, $g_i$ is increasing.
\end{subclaim}
\begin{subclaimproof}
This is straightforward and follows, by downwards induction, that for any $\beta\in S_i$, if $g_i(\beta)=(x_0,\dots,x_{k-1})$ then the maximal $l<k$ such that $x_l\neq 0^-$ satisfies that $x_l<\alpha$. 

The fact that the $g_i$s are increasing now follows by downwards induction.
\end{subclaimproof}

The main observation is that for any $1\leq i<k$
\[(\dagger)\, \otp(\langle a_\beta: \beta\in S_i\rangle, \langle b_\beta: \beta\in S_{i-1}\rangle)=\otp(\langle g_i(\beta): \beta\in S_i\rangle, \langle g_{i-1}(\beta): \beta\in S_{i-1}\rangle ).\]
To that end, let $1\leq i<k$. Since $g_i$ and $g_{i-1}$ are increasing it is enough to compare $a_{\beta_1}$ and $b_{\beta_2}$, where $\beta_1\in S_i$ and $\beta_2\in S_{i-1}$.

\begin{list}{•}{}
\item If $a_{\beta_1}=b_{\beta_2}$ then $\beta_1\in S_i$ is the minimal such that $b_{\beta_2}\leq a_{\beta_1}$ and thus by definition $g_{i-1}(\beta_2)=g_i(\beta_1)$.
\item Assume $a_{\beta_1}<b_{\beta_2}$. If there does not exist a minimal $\gamma\in S_i$ with $b_{\beta_2}\leq a_\gamma$  then
\[g_{i-1}(\beta_2)=(\pi_{i}(g_i(\infty)),\beta_2,0^-,\dots,0^-)>\]\[(\pi_{i}(g_i(\beta_1)),0^-,\dots,0^-)=g_{i}(\beta_1).\]
Otherwise, let $\beta_1< \gamma\in S_i$ be minimal such that $b_{\beta_2}\leq a_\gamma$. If $b_{\beta_2}=a_\gamma$ then $g_{i-1}(\beta_2)=g_i(\gamma)>g_i(\beta_1).$

If $b_{\beta_2}<a_\gamma$ then 
\[g_{i-1}(\beta_2)=(\pi_{i}(g_i(\gamma))^-,\beta_2,0^-,\dots, 0^-)>\]\[ (\pi_{i}(g_i(\beta_1)),0^-,0^-,\dots, 0^-)=g_i(\beta_1).\]
\item Assume $a_{\beta_1}>b_{\beta_2}$ and let $\gamma\in S_i$ be minimal such that $a_\gamma\geq b_{\beta_2}$, so $\gamma\leq \beta_1$. If $a_\gamma=b_{\beta_2}$ then $\gamma<\beta_1$ and
$g_{i-1}(\beta_2)=g_i(\gamma)<g_i(\beta_1).$

If $a_\gamma>b_{\beta_2}$ then
\[g_{i-1}(\beta_2)=(\pi_{i}(g_i(\gamma))^-,\beta_2,0^-,\dots, 0^-)<\]\[(\pi_{i}(g_i(\beta_1)),0^-,0^-,\dots, 0^-)=g_i(\beta_1).\]

\end{list}
This proves $(\dagger)$. Let $\eta,\rho\in \LSh_k(\delta)$ be as in the statement of the lemma. We proceed to prove that $f_{\eta,\bar g} \D_{\bar a,\bar b} f_{\rho,\bar g}$.

Let $\beta_1,\beta_2< \alpha$ and assume that $\beta_1\in S_{n_1}$ and $\beta_2\in S_{n_2}$, for some $0\leq n_1,n_2\leq k-1$. Note that if $b_{\beta_2}\in C_n$, for $0<n\leq k$, then $n_2=n-1$. Assume that $k>1$.
\begin{list}{•}{}
\item Assume that $0<n_1<k$, $b_{\beta_2}\in C_{n_1}$. So $n_2=n_1-1$. Assume that $a_{\beta_1}\mathrel{\square} b_{\beta_2}$, where $\square\in\left\{ <,>,=\right\}$. By $(\dagger)$, $g_{n_1}(\beta_1)\mathrel{\square} g_{n_2}(\beta_2)$ and as a result
\[f_{\eta,\bar g}(\beta_1)=(\eta(n_1),g_{n_1}(\beta_1))=(\rho(n_1-1),g_{n_1}(\beta_1))\mathrel{\square} (\rho(n_1-1),g_{n_2}(\beta_2))=\]\[(\rho(n_2),g_{n_2}(\beta_2))=f_{\rho,\bar g}(\beta_2).\]
\item If $b_{\beta_2}\in C_n$ for some $n_1<n<k$ then necessarily, $a_{\beta_1}<b_{\beta_2}$ and $n_2=n-1\geq n_1$. Consequently,
\[f_{\eta,\bar g}(\beta_1)=(\eta(n_1),g_{n_1}(\beta_1))< (\eta(n_1+1),g_{n_2}(\beta_2))=(\rho(n_1),g_{n_2}(\beta_2))\leq \]\[(\rho(n_2),g_{n_2}(\beta_2))=f_{\rho,\bar g}(\beta_2).\]

\item If $b_{\beta_2}\in C_n$ for some $n<n_1$ then necessarily $0<n<n_1$, $a_{\beta_1}>b_{\beta_2}$ and $n_2=n-1<n_1-1$. Hence
\[f_{\eta,\bar g}(\beta_1)=(\eta(n_1),g_{n_1}(\beta_1))= (\rho(n_1-1),g_{n_1}(\beta_1))>(\rho(n_2),g_{n_2}(\beta_2))=\]\[f_{\rho,\bar g}(\beta_2).\]

\item If $b_{\beta_2}\in C_k$ then necessarily $n_2=k-1$ and $a_{\beta_1}<b_{\beta_2}$.
As a result
\[f_{\eta,\bar g}(\beta_1)=(\eta(n_1),g_{n_1}(\beta_1))\leq (\eta(k-1),g_{n_1}(\beta_1))=(\rho(k-2),g_{n_1}(\beta_1))<\]
\[(\rho(k-1),g_{n_2}(\beta_2))=(\rho(n_2),g_{n_2}(\beta_2))=f_{\rho,\bar g}(\beta_2).\]

\end{list}
If $k=1$ then $n_1=n_2=0$ and
\[f_{\eta,\bar g}(\beta_1)=(\eta(0),g_{n_1}(\beta_1))<(\rho(0),g_{n_2}(\beta_2))=f_{\rho,\bar g}(\beta_2).\]
\end{claimproof}
%
%

We may now define a map $\varphi: \Sh_k(\delta)\to (I^{\U{\alpha}})_<$ by letting for $\eta\in \Sh_k(\delta)$, $\varphi(\eta)=f_{\eta,\bar g}\in (I^{\U \alpha})_<$. This maps satisfies the requirements by the previous claim.
\end{proof}

\begin{definition}\label{D:k-ord-covered}
Let $(I,<)$ and $(J,<)$ be linearly ordered sets and $\bar a, \bar b\in (I^{\U J})_<$ be increasing sequences. 

We say that $\{\bar a,\bar b\}$ is \emph{k-orderly covered} if there exists an increasing partition of $J$ into convex sets $\langle J_\varepsilon: \varepsilon\in S\rangle$ for some $S\subseteq J$, such that for every $\varepsilon\in S$, exactly one of the following holds
\begin{enumerate}
\item $\langle \bar{a}\restriction J_\varepsilon,\bar{b}\restriction J_\varepsilon\rangle$ is $k_\varepsilon$-orderly for some $0<k_\varepsilon\leq k$;
\item $\langle \bar{b}\restriction J_\varepsilon,\bar{a}\restriction J_\varepsilon\rangle$ is $k_\varepsilon$-orderly for some $0<k_\varepsilon\leq k$; 
\item $|J_\varepsilon|=1$ and $\bar{a}\restriction J_\varepsilon=\bar{b}\restriction J_\varepsilon$.
\end{enumerate}
Moreover, for every $\varepsilon<\varepsilon^\prime\in S$, $\Img (\bar a \restriction J_\varepsilon )<\Img(\bar b \restriction J_{\varepsilon^\prime})$ and $\Img(\bar b\restriction J_\varepsilon) <\Img(\bar a \restriction J_{\varepsilon^\prime})$.
\end{definition}

\begin{corollary}\label{C:k-ord-cov-implies shift}
Let $\alpha$ be an ordinal, $(I,<)$ any infinite linearly ordered set with $(|\alpha|^+ +\aleph_0,<)\subseteq (I,<)$. Let $\bar a\neq \bar b\in (I^{\U{\alpha}})_<$ be some fixed sequences.
If $\{\bar a,\bar b\}$ is $k$-orderly covered then $((I^{\U{\alpha}})_<,E_{\bar a,\bar b})$ contains all finite subgraphs of $\Sh_m(\omega)$ for some $m\leq k$.
\end{corollary}
\begin{proof}
Let $\langle J_\varepsilon: \varepsilon\in S\rangle$ be an increasing partition of $\alpha$ as in Definition \ref{D:k-ord-covered}, where $S\subseteq \alpha$. Since $\bar a\neq \bar b$, there exists $\varepsilon\in S$ such that $|J_\varepsilon|>1$.

For any $\varepsilon\in S$, with $|J_\varepsilon|>1$, we say that $J_\varepsilon$ is
\begin{list}{•}{}
\item of type $A$ if $\langle \bar a\restriction J_\varepsilon,\bar b\restriction J_\varepsilon\rangle$ is $k_\varepsilon$-orderly, and 
\item of type $B$ if $\langle \bar b\restriction J_\varepsilon,\bar a\restriction J_\varepsilon\rangle$ is $k_\varepsilon$-orderly.
\end{list}

Let $N<\omega$ be some natural number. By replacing $I$ with an isomorphic copy, we may assume that $(\alpha \times (N\times (2\alpha +1)^k),<_{lex})\subseteq (I,<)$.  Let $\varepsilon\in S$ and let $I_\varepsilon=\{\varepsilon\}\times (N\times (2\alpha+1)^{k})$. 

If $|J_\varepsilon|=1$ then we let $\varphi_\varepsilon:\LSh_1 (N)\to ((I_\varepsilon)^{\U{J_\varepsilon}})_<$ be such that $\varphi_\varepsilon(\eta)$ is the constant function giving $(\varepsilon,0,\dots,0)$.

For any $\varepsilon\in S$ let $E^\varepsilon_{\bar a,\bar b}= E_{\bar a\restriction J_\varepsilon, \bar b\restriction J_\varepsilon}$ and $D^\varepsilon_{\bar a,\bar b}= D_{\bar a\restriction J_\varepsilon, \bar b\restriction J_\varepsilon}$ and similarly $E^\varepsilon_{\bar b,\bar a}$ and $D^\varepsilon_{\bar b,\bar a}$.

If $|J_\varepsilon|>1$ and $J_\varepsilon$ is of type $A$ then let $\varphi_\varepsilon:\LSh_{k_\varepsilon}(N)\to(((I_\varepsilon)^{\U{J_\varepsilon}})_<,D^\varepsilon_{\bar a,\bar b})$ be as supplied by Lemma \ref{L:k-orderly}. I.e., for any $\eta,\rho\in (N^{\U{k_\varepsilon}})_<$, if $\eta(i)=\rho(i-1)$ for $0<i<k_\varepsilon$ (if $k_\varepsilon>1$) and $\eta(0)<\rho(0)$ (if $k_\varepsilon=1$) then $\otp(\varphi_\varepsilon(\eta),\varphi_\varepsilon(\rho))=\otp(\bar a\restriction J_\varepsilon,\bar b\restriction J_\varepsilon)$. 

If $|J_\varepsilon|>1$ and $J_\varepsilon$ is of type $B$ then let $\widehat{\varphi_\varepsilon}:\LSh_{k_\varepsilon}(N)\to (((I_\varepsilon)^{\U{J_\varepsilon}})_<,D^\varepsilon_{\bar b,\bar a})$ be as supplied by Lemma \ref{L:k-orderly}. I.e, for any $\eta,\rho\in (N^{\U{k_\varepsilon}})_<$, if $\eta(i)=\rho(i-1)$ for $0<i<k_\varepsilon$ (if $k_\varepsilon>1$) and $\eta(0)<\rho(0)$ (if $k_\varepsilon=1$) then $\otp(\widehat{\varphi_\varepsilon}(\eta),\widehat{\varphi_\varepsilon}(\rho))=\otp(\bar b\restriction J_\varepsilon,\bar a\restriction J_\varepsilon)$.  

By composing with the isomorphism $\RSh_{k_\varepsilon}(N)\to \LSh_{k_\varepsilon}(N)$ mapping $(x_0,\dots, x_{k_\varepsilon-1})$ to $(N-1-x_{k_\varepsilon-1},\dots,N-1-x_0)$, we arrive to a directed graph homomorphism $\varphi_\varepsilon:\RSh_{k_\varepsilon}(N)\to (((I_\varepsilon)^{\U{J_\varepsilon}})_<,D^\varepsilon_{\bar b,\bar a})$. By definition this map can be seen as a directed graph homomorphism $\varphi_\varepsilon:\LSh_{k_\varepsilon}(N)\to (((I_\varepsilon)^{\U{J_\varepsilon}})_<,D^\varepsilon_{\bar a,\bar b})$.

For $1\leq m\leq k$ let $\pi_m:(N^{\U k})_<\to (N^{\U m})_<$ be the projection on the first $m$ coordinates. Note that it is a directed graph homomorphism $\LSh_k(N)\to \LSh_m(N)$. We now define $\varphi:\LSh_k(N)\to ((I^{\U{\alpha}})_<,D_{\bar a,\bar b})$. For any $\eta\in (N^{\U k})_<$, $\varepsilon\in S$ and $\beta< \alpha$, let $\varepsilon(\beta)\in S$ be such that $\beta\in J_\varepsilon$. We define
\[\varphi(\eta)(\beta)= (\varepsilon(\beta),\varphi_{\varepsilon(\beta)}(\pi_{k_\varepsilon}(\eta))(\beta)).\]
Since, for any $\varepsilon\in S$ and $\eta\in (N^{\U k})_<$, $\varphi_{\varepsilon}(\pi_{k_\varepsilon}(\eta))$ is increasing, it is clear that $\varphi(\eta)$ is increasing as well.

Assume that $\eta,\rho\in \LSh_k(N)$ are connected, i.e., $\eta(i)=\rho(i-1)$ for $0<i<k$ (if $k>1$) and $\eta(0)<\rho(0)$ (if $k=1$). It is routine to check that $\otp(\varphi(\eta),\varphi(\rho))=\otp(\bar a,\bar b)$. As a result, $\varphi$ is also a graph homomorphism between $\Sh_k(N)$ to $((I^{\U \alpha})_<,E_{\bar a,\bar b})$.

We have proved that for every $N<\omega$ there exists a graph homomorphism $\varphi_N:\Sh_k(N)\to ((I^{\U \alpha})_<,E_{\bar a,\bar b})$. By compactness, we may find a graph homomorphism $\Sh_k(\omega)\to \mathcal{H}$ for some elementary extension $((I^{\U \alpha})_<,E_{\bar a,\bar b})\prec (\mathcal{H},E)$. By Fact \ref{F:homomorphism is enough}, there exists $m\leq k$ such that $((I^{\U \alpha})_<,E_{\bar a,\bar b})$ contains all finite subgraphs of $\Sh_m(\omega)$.
\end{proof}

\subsection{Analyzing order-type graphs with large chromatic number}\label{ss:orde-type-large-chrom}
The main goal of this section is to prove that every order-type graph of large enough chromatic number is $k$-orderly covered for some $k$, i.e. we will prove the following.
\begin{theorem}\label{T:shift graphs in order-type-graphs}
Let $\alpha$ be an ordinal, $(\theta,<)$ an infinite ordinal with $|\alpha|^+ +\aleph_0<\theta$. Let $\bar a\neq \bar b\in (\theta^{\U{\alpha}})_<$ be some fixed sequences.

Let $G=((\theta^{\U{\alpha}})_<,E_{\bar a,\bar b})$. If $\chi(G)>\beth_2(\aleph_0)$ then $G$  contains all finite subgraphs of $\Sh_m(\omega)$ for some $m\in \mathbb{N}$.
\end{theorem}

In order to achieve this we will need to analyze the order-type of two infinite sequences. The tools developed here, we believe, may be useful in their own right.

We fix some ordinals $\alpha$ and $\theta$ with $\theta$ infinite and $\bar a\neq \bar b\in (\theta^{\U{\alpha}})_<$ increasing sequences.

We partition $\alpha=J_0\cup J_+\cup J_-$, where
\[J_0=\{\beta< \alpha: a_\beta=b_\beta\},\, J_+=\{\beta < \alpha: a_\beta<b_\beta\},\, J_-=\{\beta< \alpha: b_\beta<a_\beta\}.\]
Let $R$ be the minimal convex equivalence relation on $\alpha$ containing
\[\{(\beta,\gamma): a_\beta=b_\gamma\},\, \{(\beta,\gamma): a_\beta<a_\gamma\leq b_\beta\} \text{ and } \{(\beta,\gamma): b_\beta< b_\gamma\leq a_\beta\}.\]

\begin{lemma}\label{L:equiv classes are disjoint}
Let $A,B\in \alpha/R$ and assume that $A<B$. Then $\Img(\bar a \restriction A)< \Img(\bar b\restriction B)$ and $\Img(\bar b\restriction A)<\Img(\bar a \restriction B)$.
\end{lemma}
\begin{proof}
We will show that $\Img(\bar a \restriction A)< \Img(\bar b\restriction B)$, the other assertion follows similarly. Let $\beta\in A$ and $\gamma\in B$, so $\beta<\gamma$.  If $a_\beta\geq b_\gamma$ then $b_\beta< b_\gamma\leq a_\beta$ and hence $\beta \R \gamma$, contradiction. 
%
%
\end{proof} 

\begin{lemma}\label{L:each class has equal direction}
\begin{enumerate}
\item For any $\beta\in J_0$, $[\beta]_R\subseteq J_0$;
\item For any $\beta\in J_+$, $[\beta]_R\subseteq J_+$;
\item For any $\beta\in J_-$, $[\beta]_R\subseteq J_-$.
\end{enumerate}

Moreover, $[\beta]_R=\{\beta\}$ for $\beta\in J_0$.
\end{lemma}
\begin{proof}
To prove (1), (2) and (3) it is sufficient to prove a weaker version where we assume that $\beta=\min [\beta]_R$.

We show $(2)$, items $(1)$ and $(3)$ are proved similarly. Assume that $[\beta]_R\not\subsetneq J_+$. Let $X= \{ \delta<\alpha: \delta\in [\beta]_R\wedge (\forall \beta\leq x\leq \delta)(a_x<b_x)\}$ (in (1) we replace $a_x<b_x$ by $a_x=b_x$ and in (3) by $a_x>b_x$). By the assumptions, $X$ is a non empty initial segment of $[\beta]_R$ and $Y=[\beta]_R\setminus X$ is non-empty convex.

We will show that both $X$ and $Y$ are closed under the relations defining $R$ and thus derive a contradiction to the minimality of $R$.

Assume that $a_y=b_z$ with $y\in X$ and $z\in Y$. Since $X$ is an initial segment, $y<z$. Consequently, $a_y<b_y<b_z$, contradiction. Now assume that $z\in X$ and $y\in Y$, so there exists $z<x\leq y$ with $a_x\geq b_x$ and as a result $a_z<b_z<b_x\leq a_x\leq a_y=b_z$, contradiction.

Assume that $a_y<a_z\leq b_y$ with $y\in X$ and $z\in Y$, so $y<z$. Hence there is some $y<x\leq z$ with $a_x\geq b_x$ hence $a_y<a_z\leq b_y<b_x\leq a_x\leq a_z$, contradiction. Now assume that $z\in X$ and $y\in Y$, so $z<y$. This implies that $a_z<a_y<a_z$, contradiction.

Assume that $b_y<b_z\leq a_y$ with $y\in X$ and $z\in Y$. Consequently, $a_y<b_y<b_z\leq a_y$, contradiction.
 Now assume that $z\in X$ and $y\in Y$, so $z<y$.  As a result, $b_z<b_y<b_z$ 
 
Finally, we show the moreover. Assume note, so by $(1)$ it is easy to see that both $\{\beta\}$ and $[\beta]_R\setminus \{\beta\}$ are closed under the relations generating $R$. This contradicts the minimality of $R$.
\end{proof}

By Lemma \ref{L:each class has equal direction}, $R\restriction J_+$ is an equivalence relation on $J_+$. For any $A\in J_+/R$ we construct a set $Z_A\subseteq A$. We construct a sequence $\delta^A_n$ for $n<\omega$ as follows. Let $\delta^A_0=\min A$  and assume that $\delta^A_n$ has been chosen. Let $\delta^A_{n+1}\in A$ be the minimal index satisfying $b_{\delta^A_n}\leq a_{\delta^A_{n+1}}$ if such exists, otherwise stop. Let $Z_A=\langle \delta^A_n: n<n_A\rangle$, where $n_A\leq \omega$. Note that $Z_A$ is a strictly increasing sequence because $A\in J_+/R$. Furthermore, set  $C^A=Conv(\Img(\bar a\restriction A)\cup \Img(\bar b\restriction A))$ and
\begin{itemize}
\item $C^A_0=[a_{\delta^A_0},b_{\delta^A_0})$;
\item If $n_A=\omega$ then for any $0<n<\omega$ we set $C^A_n=[b_{\delta^A_{n-1}},b_{\delta^A_{n}})$;
\item If $n_A<\omega$ then for any $0<n<n_A$ set $C^A_n=[b_{\delta^A_{n-1}},b_{\delta^A_{n}})$ and $C^A_{n_A}=\{x\in C^A: b_{\delta^A_{n_A-1}}\leq x\}$.
\end{itemize}
%

\begin{lemma}\label{L:we are k-orderely covered}
Let $A\in J_+/R$.
\begin{enumerate}
\item If $n_A=\omega$ then $A=\bigcup_{n<\omega} [\delta^A_0,\delta^A_n]$. 
\item If $n_A=\omega$ then $C^A=\bigcup_{n<\omega}C_n^A$.
\item For every $\beta\in A$ and $0\leq n<n_A$, $a_\beta\in C_n^A \iff b_\beta\in C_{n+1}^A$.
\end{enumerate}
\end{lemma}
\begin{proof}
\begin{enumerate}
\item Let $X=\bigcup_{n<\omega}[\delta^A_0,\delta^A_n]$. Since the $\delta^A_n$'s are chosen from $A$ and $A$ is convex, $X\subseteq A$. As in the proof of Lemma \ref{L:each class has equal direction}, it is enough to show that that both $X$ and $Y=A\setminus X$ are closed under the relations defining $R$.

If $x,y\in A$ satisfy that $a_x=b_y$ then since $b_y=a_x<b_x$ we conclude that $y<x$ and thus if $x\in X$ then $y\in X$. Now if we assume that $y\in X$, e.g. $y<\delta^A_n$, then $a_x=b_y<b_{\delta^A_n}\leq a_{\delta^A_{n+1}}$ so $x<\delta^A_{n+1}$.

Assume that $x,y\in A$ satisfy $a_x<a_y\leq b_x$. If $x\in X$, e.g. $x<\delta^A_n$, then $a_x<a_y\leq b_x<b_{\delta^A_n}\leq a_{\delta^A_{n+1}}$ so $y<\delta^A_{n+1}$. If $y\in X$ then since $x<y$ we conclude that $x\in X$ as well.

Assume that $x,y\in A$ satisfy $b_x<b_y\leq a_x$. If $y\in X$ then since $x<y$ it follows that $x\in X$ as well. Assume that $y\in Y$, i.e. $y\geq \delta^A_n$ for all $n$. But then $b_{\delta^A_n}\leq b_y\leq a_x<b_x$ for all $n$. This implies that $\delta^A_n< x$ for all $n$ and hence $x\notin X$ as well.

\item The right-to-left inclusion is straightforward. For the other inclusion, let $x\in C^A$. Since $\delta^A_0=\min A$ and $A\in J_+/R$, $a_{\delta^A_0}=\min C^A$ and hence $a_{\delta^A_0}\leq x$. If there exists $n<\omega$ with $x<b_{\delta^A_n}$ then for the minimal such $n$, $x\in C^A_{n}$. Otherwise, since $a_\beta<b_\beta$ for any $\beta\in A$, we may assume that $x\leq b_\beta$ for some $\beta\in A$. Hence $x\leq b_{\delta^A_n}$ for some $n<\omega$ by $(1)$. 

\item Let $\beta\in A$ and $n$ be as in the statement. Assume that $n_A=\omega$ is infinite ($n_A<\omega$ is similar).

Let $a_\beta\in C_n^A$. First assume $n=0$, i.e. $a_{\delta^A_0}\leq a_\beta<b_{\delta^A_0}$. It is always true that $b_{\delta^A_0}\leq b_\beta$. If $\beta\geq \delta^A_1$ then $a_{\delta^A_1}\leq a_\beta<b_{\delta^A_0}$, contradicting the choice of $\delta^A_1$.

Now, if $n>0$ then $b_{\delta^A_{n-1}}\leq a_\beta<b_{\delta^A_{n}}$ and thus by definition of $\delta^A_n$, $\delta^A_{n}\leq \beta$ so $b_{\delta^A_{n}}\leq b_\beta$. If, on the other hand, $\beta\geq \delta^A_{n+1}$ then $a_{\delta^A_{n+1}}\leq a_\beta<b_{\delta^A_{n}}$, contradiction. Hence $b_{\delta^A_{n}}\leq b_\beta<b_{\delta^A_{n+1}}$. 

Let $b_\beta\in C^A_{n+1}$. By $(2)$, $a_\beta\in C^A_k$ for some $k<\omega$. Using the above we conclude that $b_\beta\in C^A_{k+1}$ and thus $k+1=n+1$, i.e. $k=n$.
%
%
%
\end{enumerate}
\end{proof}

\begin{lemma}\label{L:special-sequence}
For any $A\in J_+/R$ there exist an increasing sequence $\langle \zeta_n^A\in A:n<n_A\rangle$, satisfying that for every $n$ with $n+1<n_A$
\[a_{\zeta^A_{n+1}}\leq b_{\zeta^A_{n}},\]
and for every $n$ with $n+2<n_A$
\[ b_{\zeta^A_{n}}<a_{\zeta^A_{n+2}}.\]
\end{lemma}
\begin{proof}
%
%
%
%
Let $n$ be such that $n+1<n_A$. Assume for now that $b_{\delta_n}\neq a_{\delta_{n+1}}$ (and hence $b_{\delta_n}<a_{\delta_{n+1}}$) and assume towards a contradiction that 
\begin{itemize}
\item[(*)] for any $\epsilon\in (\delta_n,\delta_{n+1})$, $b_{\epsilon}<a_{\delta_{n+1}}$.
\end{itemize}
 Note that this implies that for any such $\varepsilon$, $a_{\delta_n}<a_\varepsilon<b_{\delta_{n}}<b_{\varepsilon}<a_{\delta_{n+1}}$. Let $X=\{\beta \in A:\beta <\delta_{n+1}\}$ and $Y=A\setminus X$. This gives a convex partition of $A$, we will show that both $X$ and $Y$ are closed under the relations defining $R$.

Let $\beta,\gamma\in A$ with $a_\beta=b_\gamma$. If $\beta<\delta_{n+1}$ and $\gamma\geq \delta_{n+1}$ then $b_{\delta_{n+1}}\leq b_\gamma=a_\beta< a_{\delta_{n+1}}$, contradiction. Now assume that $\gamma< \delta_{n+1}$ and $\beta\geq \delta_{n+1}$. If $\gamma\leq \delta_n$ then $a_{\delta_{n+1}}\leq a_\beta=b_\gamma\leq b_{\delta_n}$, contradiction. If $\gamma>\delta_n$ then $a_{\delta_{n+1}}\leq a_\beta=b_\gamma<a_{\delta_{n+1}}$ since $\gamma\in (\delta_n,\delta_{n+1})$ and by (*), contradiction.

Let $\beta,\gamma\in A$ with $a_\beta<a_\gamma \leq b_\beta$. Assume that $\beta< \delta_{n+1}$ and $\gamma\geq \delta_{n+1}$. If $\beta\leq \delta_n$ then  $a_{\delta_{n+1}}\leq a_\gamma\leq b_\beta\leq b_{\delta_{n}}$, contradiction. If $\beta\in (\delta_n,\delta_{n+1})$ then $b_\beta<a_{\delta_{n+1}}\leq a_\gamma\leq b_\beta$ by $(*)$, contradiction. Note that we cannot have $\gamma<\delta_{n+1}$ and $\beta\geq \delta_{n+1}$ since $\beta<\gamma$ by assumption.

Let $\beta,\gamma\in A$ with $b_\beta<b_\gamma\leq a_\beta$. If $\beta< \delta_{n+1}$ and $\gamma\geq \delta_{n+1}$ then $b_{\delta_{n+1}}\leq b_\gamma\leq a_\beta<a_{\delta_{n+1}}<b_{\delta_{n+1}}$, contradiction. As before, $\gamma <\delta_{n+1}$ and $\beta\geq \delta_{n+1}$ is not possible since $\beta<\gamma$.

As a result, we may conclude that for all $n$ such that $n+1<n_A$ we may find $\gamma_n\in (\delta_n,\delta_{n+1}]$ satisfying $a_{\delta_n}<a_{\gamma_n}\leq b_{\delta_n}\leq a_{\delta_{n+1}}\leq b_{\gamma_n}$ (if $b_{\delta_n}=a_{\delta_{n+1}}$ choose $\gamma_n=\delta_{n+1}$, otherwise use the above).

Let $I=\{\delta_n,\gamma_n: n+1<n_A\}$. The crucial property is  that for every $\gamma\in I\setminus \{\sup I\}$ there is some $\beta \in I$ satisfying $a_\gamma<a_\beta\leq b_\gamma$. We note that for every $n$ such that $n+1<n_A$ if $\gamma\leq \delta_n$ then $\beta\leq \delta_{n+1}$.  Indeed, otherwise $a_{\delta_{n+1}}<a_\beta\leq b_\gamma\leq b_{\delta_n}$, contradiction.

We construct a sequence $\langle \zeta_n: n<k\rangle$ for some $k\leq \omega$ as follows. Define $\zeta_0=\delta_0$ and for every $n$ let $\zeta_{n+1}\in I$ be maximal\footnote{If $n_A$ is finite then such a maximal element clearly exists. Otherwise, for $\zeta \in I$ there is some $n<\omega$ such that $b_{\zeta}<a_{\delta_n}$, and hence $\zeta\in \{\xi\in I: a_\xi\leq b_\zeta\}$ is finite.} with $a_{\zeta_{n}}<a_{\zeta_{n+1}}\leq b_{\zeta_{n}}$, if exists. Obviously, this is an increasing sequence. We claim that $k\geq n_A$. By induction on $n<k$ with $n<n_A$, $\zeta_n\leq \delta_n$. In particular if $n+1<n_A$, $\zeta_{n+1}$ exists. Finally we note that by maximality, for all $n+2<n_A$, $a_{\zeta_{n+1}}\leq b_{\zeta_{n}}< a_{\zeta_{n+2}}$.
\end{proof}

For $A\in J_-/R$ we make dual (i.e. exchanging the roles of $\bar a$ and $\bar b$) constructions and similar properties hold.

\begin{corollary}\label{C:embed shift if k bounds n_A}
Let $\alpha$ be an ordinal, $(\theta,<)$ an infinite ordinal with $|\alpha|^+ +\aleph_0<\theta$. Let $\bar a\neq \bar b\in (\theta^{\U{\alpha}})_<$ be some fixed sequences.

Let $G=((\theta^{\U{\alpha}})_<,E_{\bar a,\bar b})$. If there exists $0<k<\omega$ with $n_A\leq k$ for all $A\in (J_+\cup J_-)/R$ then $G$ contains all finite subgraphs of $\Sh_m(\omega)$ for some $m\leq k$.
\end{corollary}
\begin{proof}
By Lemma \ref{L:equiv classes are disjoint} and Lemma \ref{L:we are k-orderely covered}(3), $\{\bar a,\bar b\}$ is $k$-orderly covered in the sense of Definition \ref{D:k-ord-covered}. Now apply Corollary \ref{C:k-ord-cov-implies shift}.
\end{proof}

The aim of the rest of this section is to prove that $\{n_A: A\in (J_+\cup J_-)/R\}$ has a finite bound. From now on we will only need the sequences defined in Lemma \ref{L:special-sequence}.

%
%
\begin{lemma}\label{L:each n_A is finite}
If $\chi(G)>2^{\aleph_0}$ then for any $A\in (J_+\cup J_-)/R$, $n_A<\omega$.
\end{lemma}
\begin{proof}
We assume that $A\in J_+/R$, the proof for $A\in J_-/R$ is similar. Assume towards a contradiction that $n_A=\omega$. We will show that $\chi(G)\leq 2^{\aleph_0}$.

Let $S=\{\beta\leq \theta:\cf(\beta)=\aleph_0\}$. 
%
Let $\langle \zeta_l=\zeta^A_l\in A:l<\omega\rangle$ be the sequence supplied by Lemma \ref{L:special-sequence}. 

For any $\gamma\in S$ choose an increasing sequence of ordinals $\langle \alpha_{\gamma,n}:n<\omega\rangle\subseteq \gamma$ with limit $\gamma$. We define a coloring map $c:(\theta^{\U{\alpha}})_<\to 2^{\aleph_0\times\aleph_0}$.
For any $\bar f\in (\theta^{\U{\alpha}})_<$ let $\gamma(\bar f)=\sup\{f_{\zeta_l}:l<\omega\}\in S$ and
\[c(\bar f)=\{(l,n):l,n<\omega,\, f_{\zeta_l}<\alpha_{\gamma(\bar f),n}\}.\]
To show that it is a legal coloring, let $\bar f,\bar g\in (\theta^{\U{\alpha}})_<$ such that $\otp(\bar f,\bar g)=\otp(\bar a,\bar b)$.
By assumption $f_{\zeta_l}<f_{\zeta_{l+1}}\leq g_{\zeta_l}\leq f_{\zeta_{l+2}}$ for $l<\omega$, and hence $\gamma(\bar f)=\gamma(\bar g)$. By definition, there is some $n<\omega$ such that $f_{\zeta_0}<\alpha_{\gamma(\bar f),n}$ and let $k$ be the minimal such that $f_{\zeta_{k+1}}\geq \alpha_{\gamma(\bar f),n}$. So by minimality of $k$,
\[f_{\zeta_k}<\alpha_{\gamma(\bar f),n}\leq f_{\zeta_{k+1}}\leq g_{\zeta_{k}}\]
and hence $(k,n)\in c(\bar f)$ but $(k,n)\notin c(\bar g)$ so $c(\bar f)\neq c(\bar g)$.
%
%
%
%
\end{proof}

The next lemma requires a more complicated argument: Section \ref{S:PCF} below. Let us introduce some notation. 

Fix some sequence $\langle A_\varepsilon \in (J_+/R : \varepsilon<\omega \rangle$. 
For any $\varepsilon<\omega$, let $J_\varepsilon=\{ \zeta_n^{A_\varepsilon}\in A_\varepsilon: n<n_{A_\varepsilon}\}$ be the sequence supplied by Lemma \ref{L:special-sequence} applied to $A_\varepsilon$. 
Let $J=\bigcup_{\varepsilon<\omega}J_\varepsilon$,
 $\Omega=(\theta^{\U J})_<$, 
 $R'=\{(\bar c,\bar d)\in \Omega^2: \otp(\bar c,\bar d)=\otp(\bar a\restriction J,\bar b\restriction J)\}$ and 
 $\chi=\beth_2(\aleph_0)$. 
 Note that $R'$ is an irreflexive relation on $\Omega$ satisfying that if $f_1 \R' f_2$, $f_1,f_2\in \Omega$, then for every $\varepsilon<\omega$ and $i\in J_\varepsilon$, the following hold:
\[f_1(i)<f_2(i)\tag{1}\] 
and for any $i\in J_\varepsilon$ with $\suc{i}\in J_\varepsilon$
\[f_1(\suc{i})\leq f_2(i)\tag{2},\] and for any $i\in J_\varepsilon$ with $\suc{\suc{i}}\in J_\varepsilon$,
\[f_2(i)<f_1(\suc{\suc{i}})\tag{3},\] where $\suc{i}$ is the successor of $i$ in $J_\varepsilon$.

Under these assumptions (or more generally under Assumption \ref{A:R}), we will prove in Conclusion \ref{Con:coloring-pcf} that 
\begin{itemize}
    \item[(*)] If $n_{A_\varepsilon} < \omega$ for all $\varepsilon<\omega$, then there exists a function $c:\Omega\to \chi$ satisfying that if $f_1,f_2\in \Omega$ and $f_1 \R' f_2$ then $c(f_1)\neq c(f_2)$. In other words, there exists a coloring of the directed graph $(\Omega,R')$ of cardinality $\chi$. 
\end{itemize}

\begin{lemma}\label{L:all n_A are bounded}
If $\chi(G)>\beth_2(\aleph_0)$ then the set $\{n_A: A\in (J_+\cup J_-)/R\}$ is bounded.
\end{lemma}
\begin{proof}
By Lemma \ref{L:each n_A is finite}, for any $A\in (J_+\cup J_-)/R$, $n_A<\omega$. We will show that $\{n_A: A\in J_+/R\}$ and $\{n_A: A\in J_-/R\}$ are both bounded.

Assume that $\{n_A: A\in J_+/R\}$ is unbounded. Let $\{A_\varepsilon\in J_+/R:\varepsilon<\omega\}$ be a family of convex equivalence classes such that $\varepsilon<n_{A_\varepsilon}$. 


By (*), there exists a function $c:\Omega\to \beth_2(\aleph_0)$ satisfying that if $f_1,f_2\in \Omega$ and $f_1\R' f_2$ then $c(f_1)\neq c(f_2)$. Let $H=(\Omega,(R')^{sym})$ be the graph induced by $R'$ (i.e. $(\bar c, \bar d)\in (R')^{sym} \iff (\bar c,\bar d)\in R'\vee (\bar d,\bar c)\in R'$). The map $c$ induces a coloring on $H$ and hence $\chi(H)\leq \beth_2(\aleph_0)$. Since the map $(\theta^{\U \alpha})_< \to \Omega$ given by $\eta\mapsto \eta\restriction J$ is a graph homomorphism, $\chi(G)\leq \beth_2(\aleph_0)$ and this contradicts the assumption.

If on the other hand $\{n_A: A\in J_-/R\}$ is unbounded then we proceed as above but using $R''=\{(\bar c,\bar d)\in \Omega^2: \otp(\bar c,\bar d)=\otp(\bar b\restriction J,\bar a\restriction J)\}$ and the dual construction (replacing $R'$ by $R''$ in (*)) mentioned above instead and arrive at a similar contradiction.
\end{proof}

Finally, we may conclude:
\proof[proof of Theorem \ref{T:shift graphs in order-type-graphs}]
This is a direct consequence of Corollary \ref{C:embed shift if k bounds n_A} and Lemma \ref{L:all n_A are bounded}.
\qed

\section{Coloring increasing functions}\label{S:PCF}
This section's main result is Conclusion \ref{Con:coloring-pcf}, used in the final stage of the previous section. We prove that under mild conditions on a directed graph, namely Assumption \ref{A:R}, on a family of strictly increasing functions there exists a coloring of small cardinality.

Let $\kappa=\cf(\kappa)$ be a regular cardinal and $(J,<)$ a well order of cofinality $\kappa$. Let $\sigma=(2^{\kappa})^+$, $\theta$ be an ordinal and $\chi=\chi^{\kappa}+\chi^{<\sigma}$ a cardinal. 

Let $\langle J_\varepsilon: \varepsilon <\kappa\rangle$ be an increasing partition of $J$ into  finite convex sets. Assume that $\sup_{\varepsilon<\kappa} |J_\varepsilon|=\omega$. Let $\mathcal{D}$ be a non-principal ultrafilter on $\kappa$ containing the filter generated by $\{\{\varepsilon<\kappa: |J_\varepsilon|\geq n\}: n<\omega\}$.

Let $\Omega$ be the set of functions from $J$ to $\theta$ that are strictly increasing on each $J_\varepsilon$ ($\varepsilon <\kappa$). Let $\mathcal{H}=(\theta+1)^\kappa$.

\begin{assumption}\label{A:R}
$R$ is an irreflexive relation on $\Omega$ satisfying that if $f_1 \R f_2$, $f_1,f_2\in \Omega$, then for every $\varepsilon<\kappa$ and $i\in J_\varepsilon$
\[f_1(i)<f_2(i)\tag{1}\] and for any $i\in J_\varepsilon$ with $\suc{i}\in J_\varepsilon$
\[f_1(\suc{i})\leq f_2(i)\tag{2},\]
and for any $i\in J_\varepsilon$ with $\suc{\suc{i}}\in J_\varepsilon$,
\[f_2(i)<f_1(\suc{\suc{i}})\tag{3},\]
where $\suc{i}$ is the successor of $i$ in the finite set $J_\varepsilon$.

%
\end{assumption}

We say a subset $X$ of $\Omega$ is \emph{trivial} if $f_1\not\R f_2$ for any $f_1,f_2\in X$.

\begin{definition}\label{D:def of approximation}
An approximation $\Ba$ is a partition $\Omega=\bigcup_{s\in S_\Ba}\Omega_{s}^\Ba$ (so all the $\Omega^{\Ba}_s$'s are non-empty), $\rho_{s}^\Ba\in \chi^{<\sigma}$ and $h^\Ba_s\in \mathcal{H}$ ($s\in S_\Ba$) satisfying
\begin{enumerate}
\item for every $s\in S_\Ba$ and $f\in\Omega^\Ba_s$, $\{\varepsilon<\kappa: \Rg(f\restriction J_\varepsilon)\subseteq h^\Ba_s(\varepsilon)\}\in \mathcal{D}$
\item if $s\neq t\in S_\Ba$ and $\rho_s^\Ba=\rho^\Ba_t$ then for every $f_1\in \Omega^\Ba_s$ and $f_2\in \Omega^\Ba_t$, $f_1\not\R f_2$.
\end{enumerate}
\end{definition}

We want to define when one approximation is better than the other.
\begin{definition}\label{D:order on approx}
For two approximations $\Ba$ and $\Bb$ we will say that $\Ba\trianglelefteq_g \Bb$ if there exists a surjective function $g: S_\Bb \to S_\Ba$ satisfying
\begin{enumerate}
\item for any $s\in S_\Ba$, $\{\Omega^\Bb_t: t\in g^{-1}(s)\}$ is a partition of $\Omega^\Ba_s$.
\item if $s\in S_\Ba$ and $\Omega^\Ba_s$ is trivial then $g^{-1}(s)$ is a singleton $t\in S_\Bb$ satisfying $h^\Ba_s=h^\Bb_t$ and $\Omega^\Ba_s=\Omega^\Bb_t$ (in particular $\Omega^\Ba_t$ is also trivial).
\item for $t\in S_\Bb$,  $\{\varepsilon<\kappa: h^\Bb_t(\varepsilon)\leq h^\Ba_{g(t)}(\varepsilon)\}\in \mathcal{D}$.
\item for $t\in S_\Bb$, $\rho^\Ba_{g(t)}$ is an initial segment of $\rho^\Bb_t$.
\end{enumerate}

We will say that $\Ba\triangleleft_g \Bb$ if $\Ba\trianglelefteq_g \Bb$ and in addition for every $t\in S_\Bb$, either $\Omega^\Bb_t$ is trivial or $\{\varepsilon<\kappa :h^\Bb_t(\varepsilon)<h^\Ba_{g(t)}(\varepsilon)\}\in \mathcal{D}$. 
\end{definition}

The following is clear.
\begin{lemma}\label{L:composition-approx}
Let $\Ba,\Bb$ and $\Bc$ be approximations. If $\Ba\trianglelefteq_g \Bb$ and $\Bb\trianglelefteq_h\Bc$ then $\Ba\trianglelefteq_{g\circ h} \Bc$. If, in addition, either $\Ba \triangleleft_g \Bb$ or $\Bb \triangleleft_h\Bc$ then $\Ba \triangleleft_{g\circ h}\Bc$.
\end{lemma}

\begin{proposition}\label{P:approx-small-cof}
Let $\Ba$ be an approximation. Then there exists an approximation $\Ba\trianglelefteq_g \Bb$ satisfying the following.
\begin{enumerate}
\item If $t\in S_\Bb$, $\Omega^\Ba_{g(t)}$ is non-trivial and $\{\varepsilon<\kappa: 0<\cf(h^\Ba_{g(t)}(\varepsilon))\leq \chi\}\in \mathcal{D}$ then either $\Omega^\Bb_t$ is trivial or $\{\varepsilon<\kappa: h^\Bb_t(\varepsilon)<h^\Ba_{g(t)}(\varepsilon)\}\in\mathcal{D}$.
\item If $t\in S_\Bb$ satisfies that $\{\varepsilon<\kappa:0<\cf(h^\Ba_{g(t)}(\varepsilon))\leq \chi\}\notin \mathcal{D}$ then $h^\Bb_t=h^\Ba_{g(t)}$ and $\Omega^\Bb_t=\Omega^\Ba_{g(t)}$.
\end{enumerate}
Lastly, for $t\in S_\Bb$, if $\rho^\Ba_{g(t)}\in \chi^{\xi}$, with $\xi<\sigma$, then $\rho^\Bb_t\in \chi^{\xi+1}$.
\end{proposition}
\begin{proof}
We partition $S_\Ba$ in the following way. Let $S_1=\{s\in S_\Ba:  (\forall^{\mathcal{D}}\varepsilon<\kappa)(0<\cf(h^\Ba_s(\varepsilon))\leq \chi) \text{ and } \Omega^\Ba_s \text{ is non-trivial}\}$ and $S_0=S_\Ba\setminus S_1$ is the rest.

Fix any $s\in S_1$. For any $\varepsilon<\kappa$, if $0<\cf(h^\Ba_s(\varepsilon))\leq \chi$ we choose a an unbounded subset $C_{s,\varepsilon}\subseteq h^\Ba_s(\varepsilon)$ of order type $\cf(h^\Ba_s(\varepsilon))$, and we set $C_{s,\varepsilon}=\{h^\Ba_s(\varepsilon))\}$, otherwise.

Set $A_s=\{\varepsilon<\kappa: 0<\cf(h^\Ba_s(\varepsilon))\leq \chi\}$. Note that $A_s\in \mathcal{D}$.

Let $H_s=\{h\in \mathcal{H}: \text{if $\varepsilon\in A_s$ then $h(\varepsilon)\in C_{s,\varepsilon}$ and $h(\varepsilon)=h^\Ba_s(\varepsilon)$, otherwise}\}$. Since $\chi^{\kappa}=\chi$, $|H_s|\leq \chi$ and hence there is some $\xi_s\leq \chi$ and an enumeration $\langle h_{s,\xi}:\xi<\xi_s\rangle$ of $H_s$.

By induction on $\xi<\xi_s$ we define
\[\Omega_{s,\xi}=\{f\in \Omega^\Ba_s: (\forall^\mathcal{D}\varepsilon<\kappa)(\Rg(f\restriction J_\varepsilon)\subseteq h_{s,\xi}(\varepsilon))\}\setminus \bigcup_{\alpha<\xi}\Omega_{s,\alpha}\]
and for $\xi=\xi_s$ 
\[\Omega_{s,\xi}=\{f\in \Omega^\Ba_s:(\forall^\mathcal{D} \varepsilon<\kappa)(h^\Ba_s(\varepsilon) \text{ is a sucessor and }h^\Ba_s(\varepsilon)-1=\max\Rg(f\restriction J_\varepsilon)\}.\]
We claim that $\Omega^\Ba_s=\bigsqcup_{\xi\leq \xi_s}\Omega_{s,\xi}$. Let $f\in \Omega^\Ba_s$. Note that for every $\varepsilon\in A_s$ either (a) there an ordinal $\gamma\in C_{s,\varepsilon}$ such that $\Rg(f\restriction J_\varepsilon)\subseteq \gamma$ or (b) there is no such $\gamma$. We may find $A_s^\prime\subseteq A_s$ satisfying that $A_s^\prime\in\mathcal{D}$ and that for all $\varepsilon\in A_s^\prime$ (a) holds or that for all $\varepsilon\in A_s^\prime$ (b) holds.

Assume that (a) holds for all $\varepsilon\in A_s^\prime$ and let $\langle \gamma_\varepsilon : \varepsilon<A_s^\prime\rangle$ witness this. Define a function $h\in H_s$ by setting for all $\varepsilon\in A_s^\prime$, $h(\varepsilon)=\gamma_\varepsilon$. For $\varepsilon\notin A_s^\prime$ choose arbitrary $h(\varepsilon)$ as long as $h\in H_s$. Let $\xi<\xi_s$ be minimal such that $(\forall^\mathcal{D}\varepsilon<\kappa)(\Rg(f\restriction J_\varepsilon)\subseteq h_{s,\xi}(\varepsilon))$, so $f\in \Omega_{s,\xi}$.

Now, assume that (b) holds for all $\varepsilon\in A_s^\prime$. As $\Ba$ is an approximation (see Definition \ref{D:def of approximation}(1)) we may assume that for all $\varepsilon\in A_s^\prime$, $\Rg(f\restriction J_\varepsilon)\subseteq h^\Ba_s(\varepsilon)$ but we cannot find any $\gamma\in C_{s,\varepsilon}$ satisfying $\Rg(f\restriction J_\varepsilon)\subseteq \gamma$. For any $\varepsilon \in A_s^\prime$, because $J_\varepsilon$ is finite this implies that $\cf(h_s^{\Ba}(\varepsilon))=1$, i.e. that $h^\Ba_s(\varepsilon)$ is  a successor ordinal and that $h^\Ba_s(\varepsilon)-1=\max\Rg(f\restriction J_\varepsilon)$. Hence $f\in \Omega_{s,\xi_s}$.

Let $S_\Bb=\{(s,\xi):s\in S_1,\, \xi\leq \xi_s,\, \Omega_{s,\xi}\neq \emptyset\}\cup S_0$ and let $g:S_\Bb\to S_\Ba$ be the function defined by $g(s,\xi)=s$ for $s\in S_1$ and $g(s)=s$ otherwise. For any $s\in S_0$ let $\Omega^\Bb_s=\Omega^\Ba_s$, $h^\Bb_s=h^\Ba_s$ and $\rho^\Bb_s={\rho^\Ba_s}^\frown\langle 0\rangle $. For $s\in S_1$, if $\xi\leq \xi_s$ we set $\Omega^\Bb_{(s,\xi)}=\Omega_{s,\xi}$ and $\rho^\Bb_{(s,\xi)}={\rho^\Ba_s}^{\frown} \langle \xi\rangle$. Finally, for $\xi<\xi_s$ we set $h^\Bb_{(s,\xi)}=h_{s,\xi}$ and for $\xi=\xi_s$ we set $h^\Bb_{(s,\xi)}=h^\Ba_s$.

\begin{claim}
$\Bb$ is an approximation and $\Ba\trianglelefteq_g \Bb$.
\end{claim}
\begin{claimproof}
We first show that $\Bb$ is an approximation. Items $(1)$ and $(2)$ from the definition follow since $\Ba$ is an approximation and the construction above. For example, if $(s_1,\xi_1)\neq (s_2,\xi_2)\in S_\Bb$ and $\rho^\Bb_{(s_1,\xi_1)}=\rho^\Bb_{(s_2,\xi_2)}$ then since $\xi_1=\xi_2$ necessarily $s_1\neq s_2$ and $\rho^\Ba_{s_1}=\rho^\Ba_{s_2}$ so we may use the fact that $\Ba$ is an approximation.
%
%

Finally, $\Ba\trianglelefteq_g \Bb$ by the construction.
\end{claimproof}

Showing $(1)$ from the statement of the proposition boils down to showing that $\Omega^\Bb_{(s,\xi_s)}=\Omega_{s,\xi_s}$ is trivial. This follows from Assumption \ref{A:R}(1).
\end{proof}

\begin{proposition}\label{P:approx-large-cof}
Let $\Ba$ be an approximation. Then there exists an approximation $\Ba\trianglelefteq_g \Bb$ satisfying the following. 
\begin{enumerate}
\item If $t\in S_\Bb$ with $\{\varepsilon<\kappa: \cf(h^\Ba_{g(t)}(\varepsilon))>\chi\}\in \mathcal{D}$ then either $\Omega^\Bb_t$ is trivial or $\{\varepsilon<\kappa: h^\Bb_t(\varepsilon)<h^\Ba_{g(t)}(\varepsilon)\}\in \mathcal{D}$.
\item If $t\in S_\Bb$ satisfies that $\{\varepsilon<\kappa:\cf(h^\Ba_{g(t)}(\varepsilon))> \chi\}\notin \mathcal{D}$ then $h^\Bb_t=h^\Ba_{g(t)}$ and $\Omega^\Bb_t=\Omega^\Ba_{g(t)}$.
\end{enumerate}

Lastly, for $t\in S_\Bb$, if $\rho^\Ba_{g(t)}\in \chi^{\xi}$, with $\xi<\sigma$, then $\rho^\Bb_t\in \chi^{\xi+1}$.
\end{proposition}
\begin{proof}
Let $S_1=\{s\in S_\Ba: \Omega^\Ba_s \text{ is non-trivial and } (\forall^{\mathcal{D}} \varepsilon<\kappa)(\cf(h^\Ba_s(\varepsilon))>\chi\}$ and $S_0=S_\Ba\setminus S_1$. Fix any $s\in S_1$. Let $A_s=\{\varepsilon<\kappa: \cf(h^\Ba_s(\varepsilon))>\chi\}$, so $A_s\in \mathcal{D}$.

Let $\mathcal{D}_s=\{D\in \mathcal{D}: D\subseteq A_s\}$ be the induced ultrafilter on $A_s$. Consider the ultraproduct $\prod_{\varepsilon\in A_s}h^\Ba_s(\varepsilon)/\mathcal{D}_s$. We may consider it as a linearly ordered set, ordered by $<_{\mathcal{D}_s}$.

\begin{claim}\label{C: sequence in ultraproduct}
There exists a sequence $H_s=\langle h_{s,\beta}\in \mathcal{H}:\beta<\beta_s\rangle$ satisfying
\begin{enumerate}
\item for all $\varepsilon\in A_s$ and $\beta<\beta_s$, $h_{s,\beta}(\varepsilon)<h^\Ba_s(\varepsilon)$;
\item for all $\varepsilon\in \kappa\setminus A_s$ and $\beta<\beta_s$, $h_{s,\beta}(\varepsilon)=h^\Ba_s(\varepsilon)$;
\item $\langle (h_{s,\beta}\restriction A_s)/\mathcal{D}_s:\beta<\beta_s\rangle$ is $<_{\mathcal{D}_s}$ increasing and cofinal in  $\prod_{\varepsilon\in A_s}h^\Ba_s(\varepsilon)/\mathcal{D}_s$.
\item for any $f\in \Omega^\Ba_s$ there exists $\beta<\beta_s$ such that $\{\varepsilon<\kappa: \Rg(f\restriction J_\varepsilon)\subseteq h_{s,\beta}(\varepsilon)\}\in\mathcal{D}$.
\end{enumerate}
\end{claim}
\begin{claimproof}
First we choose a well-ordered increasing cofinal sequence in $\prod_{\varepsilon\in A_s}h^\Ba_s(\varepsilon)/\mathcal{D}_s$ and then choose a sequence of representatives  $\langle h_{s,\beta}\restriction A_s:\beta<\beta_s\rangle$. To get (2), set $h_{s,\beta}(\varepsilon)=h_s^{\Ba}(\varepsilon)$ for any $\varepsilon\in \kappa\setminus A_s$. This gives us (1)--(3).

We show (4). Let $f\in \Omega_s^\Ba$. Since $\Ba$ is an approximation, the set $X_{s,f}=\{\varepsilon\in A_s: \Rg(f\restriction J_\varepsilon)\subseteq h^{\Ba}_s(\varepsilon)\}$ is in $\mathcal{D}$. Let $h_f: A_s\to \text{Ord}$ be the function defined by mapping $\varepsilon\in X_{s,f}$ to $\max \Rg(f\restriction J_\varepsilon)+1$ and $\varepsilon\in A_s\setminus X_{s,f}$ to $0$. Note that for any $\varepsilon\in A_s$, $h_f(\varepsilon)<h_s^{\Ba}(\varepsilon)$. Indeed, if for $\varepsilon\in X_{s,f}$, $h_f(\varepsilon)=h_s^{\Ba}(\varepsilon)$ then $h_s^{\Ba}(\varepsilon)$ is a successor contradicting $\varepsilon\in A_s$. Similarly (and even easier), this holds if $\varepsilon\in A_s\setminus X_{s,f}$. It follows that for some $\beta<\beta_s$, $h_f/\mathcal{D}_s\leq_{\mathcal{D}_s} (h_{s,\beta}\restriction A_s)/\mathcal{D}_s$ and it is easy to check that this $\beta$ satisfies (4).
\end{claimproof}

For any $f\in \Omega^\Ba_s$, $n<\omega$ and $\beta<\beta_s$ let $B_n(f,h_{s,\beta})=\{\varepsilon<\kappa: |\{i\in J_\varepsilon: f(i)\geq h_{s,\beta}(\varepsilon)\}|\leq n\}$. By Claim \ref{C: sequence in ultraproduct}(4), we may set $\beta_{s,n}(f)=\min \{\beta: B_n(f,h_{s,\beta})\in \mathcal{D}\}$.  Note that $\beta_{s,n}(f)\geq \beta_{s,n+1}(f)$. Let $\beta_s(f)=\min \{\beta_{s,n}(f):n<\omega\}$ and let $n_s(f)=\min \{n<\omega: (\forall k\geq n)(\beta_{s,k}(f)=\beta_{s,n}(f)\}$.

\begin{claim}\label{C:=n_s(f)}
$\{\varepsilon<\kappa: |\{i\in J_\varepsilon: f(i)\geq h_{s,\beta_s(f)}(\varepsilon)\}|=n_s(f)\}\in \mathcal{D}$.
\end{claim}
\begin{claimproof}
Call this set $Y_{s,f}$. Note that $Y_{s,f}\subseteq B_{n_s(f)}(f,h_{s,\beta_s(f)})$.

If $n_s(f)=0$ then $Y_{s,f}=B_0(f,h_{s,\beta_{s,0}(f)})\in \mathcal{D}$. Assume $n_s(f)>0$. If $Y_{s,f}\notin \mathcal{D}$ then $\{\varepsilon<\kappa: |\{i\in J_\varepsilon: f(i)\geq h_{s,\beta_s(f)}(\varepsilon)\}|\leq n_s(f)-1\}\in \mathcal{D}$. So $\beta_{s,n_s(f)-1}\leq \beta_s(f)=\beta_{s,n_s(f)}(f)$, contradiction.
\end{claimproof}

For $s\in S_1$, $\beta<\beta_s$ and $n<\omega$, let $\Omega_{(s,\beta,n)}=\{f\in \Omega^\Ba_s: \beta_s(f)=\beta,\, n_s(f)=n\}$.
Let $S_\Bb=\{(s,\beta,n):s\in S_1,\, \beta< \beta_s,\, n<\omega,\, \Omega_{(s,\beta,n)}\neq\emptyset \}\cup S_0$ and let $g:S_\Bb\to S_\Ba$ be the function defined by $g(s,\beta,n)=s$ for $s\in S_1$ and $g(s)=s$ otherwise. For any $s\in S_0$ let $\Omega^\Bb_s=\Omega^\Ba_s$, $h^\Bb_s=h^\Ba_s$ and $\rho^\Bb_s={\rho^\Ba_s}^\frown\langle 0\rangle $. For $s\in S_1$, $\beta<\beta_s$ and $n<\omega$, we set $\Omega^\Bb_{(s,\beta,n)}=\Omega_{(s,\beta,n)}$, $\rho^\Bb_{(s,\beta,n)}={\rho^\Ba_s}^\frown \langle n\rangle$ and 
\[
h^\Bb_{(s,\beta,n)}=
\begin{cases}
h_{s,\beta} & n=0 \\
h^\Ba_s & n>0
\end{cases}
.\]

\begin{claim}
$\Bb$ is an approximation and $\Ba\trianglelefteq_g \Bb$.
\end{claim}
\begin{claimproof}
We check that $\Bb$ satisfies $(1)$ and $(2)$ from the definition. $(1)$ follows by the choice of $h^\Bb_{(s,\beta,n)}$.

We are left with $(2)$. Let $t\in S_0$ and $(s,\beta,n)$ with $s\in S_1$. If $\rho^\Bb_t=\rho^\Bb_{(s,\beta,n)}$ then $\rho^\Ba_t=\rho^\Ba_s$ so the result follows since $\Ba$ is an approximation. Let $(s_1,\beta_1,n_1)\neq (s_2,\beta_2,n_2)\in S_\Bb$. If $\rho^\Bb_{(s_1,\beta_1,n_1)}=\rho^\Bb_{(s_2,\beta_2,n_2)}$ then $\rho^\Ba_{s_1}=\rho^\Ba_{s_2}$. If $s_1\neq s_2$ then the results follows since $\Ba$ is an approximation. So assume that $s=s_1=s_2$ and $n=n_1=n_2$. Assume that $\beta_1<\beta_2<\beta_s$ and let $f_1\in \Omega^\Bb_{(s,\beta_1,n)}$ and $f_2\in \Omega^\Bb_{(s_,\beta_2,n)}$. We need to show that $f_1\not\R f_2$ and $f_2\not\R f_1$.

By choice of $\beta_1=\beta_s(f_1)$ and $n=n_s(f_1)$, $B_n(f_1,h_{s,\beta_1})\in \mathcal{D}$. On the other hand, since $\beta_1<\beta_2=\beta_s(f_2)\leq \beta_{s,n+2}(f_2)$, $B_{n+2}(f_2,h_{s,\beta_1})\notin \mathcal{D}$. I.e. $\kappa\setminus B_{n+2}(f_2,h_{s,\beta_1})\in \mathcal{D}$. Let $\varepsilon\in B_n(f_1,h_{s,\beta_1})\cap (\kappa\setminus B_{n+2}(f_2,h_{s,\beta_1}))$ and let $i_l$ be the $(n+l)$-th element of $J_\varepsilon$ from the end, for $l=1,2,3$.
As a result,
\[f_1(i_3)<f_1(i_1)<h_{s,\beta_1}(\varepsilon)\leq f_2(i_3).\] Consequently, $f_1\not\R f_2$ by Assumption \ref{A:R}(3) (since $\suc{\suc{i_3}}=i_1$) and $f_2\not\R f_1$ by Assumption \ref{A:R}(1).

$\Ba\trianglelefteq_g \Bb$ by construction.
\end{claimproof}

To complete the proof, we note that for $(s,\beta,n)\in S_\Bb$ with $n>0$, $\Omega^\Bb_{(s,\beta,n)}$ is trivial. Let $f_1,f_2\in \Omega^\Bb_{(s,\beta,n)}$. By Claim \ref{C:=n_s(f)}, and the assumptions on $\mathcal{D}$, we may find $\varepsilon<\kappa$ such that for $l=1,2$ the following holds
\begin{enumerate}
\item $|\{i\in J_\varepsilon: f_l(i)\geq h_{s,\beta}(\varepsilon)\}|=n$ and
\item $|J_\varepsilon|>n$.
\end{enumerate}
Since $n>0$, letting $i$ be the $n+1$-th element from the end of $J_\varepsilon$ we have that $f_l(i)<h_{s,\beta}(\varepsilon)\leq f_l(\suc{i})$, for $l=1,2$. If $f_1\R f_2$ then $f_1(\suc{i})\leq f_2(i)<h_{s,\beta}(\varepsilon)$ by Assumption \ref{A:R}(2), contradiction.

\end{proof}

\begin{proposition}\label{P:approximation,proper-successor}
Let $\Ba$ be an approximation.  Then there exists an approximation $\Bc$ and a surjective function $r:S_\Bc\to S_\Ba$ such that $\Ba \triangleleft_r \Bc$.
Moreover, for $t\in S_\Bc$, if $\rho^\Ba_{r(t)}\in \chi^{\xi}$, with $\xi<\sigma$, then $\rho^\Bc_t\in \chi^{\xi+2}$.
\end{proposition}
\begin{proof}
Let $\Ba\trianglelefteq_g \Bb$ be the approximation supplied by Proposition \ref{P:approx-small-cof} and let $\Bb\trianglelefteq_f \Bc$ be the approximation supplied by Proposition \ref{P:approx-large-cof}.
Note that $\Ba\trianglelefteq_{g\circ f}\Bc$ by Lemma \ref{L:composition-approx}, we claim that $\Ba \triangleleft_{g\circ f} \Bc$.

Let $t\in S_\Bc$. Since $\Ba$ is an approximation and $\sup_{\varepsilon <\kappa} |J_\varepsilon|=\omega$, we cannot have that $\{\varepsilon<\kappa: \cf(h_{gf(t)}^\Ba(\varepsilon))=0\}\in \mathcal{D}$, see Definition \ref{D:def of approximation}(1). 

Assume that $\{\varepsilon<\kappa: 0<\cf(h_{gf(t)}^\Ba(\varepsilon))\leq \chi\}\in \mathcal{D}$. If $\Omega^\Ba_{gf(t)}$ is trivial then so is $\Omega^\Bc_t$, so assume not. By Proposition \ref{P:approx-small-cof}(1) applied to $f(t)\in S_\Bb$, either $\Omega_{f(t)}^\Bb$ is trivial (and thus so is $\Omega^\Bc_t$) or $\{\varepsilon<\kappa: h_{f(t)}^\Bb<h^\Ba_{gf(t)}(\varepsilon)\}\in \mathcal{D}$. If it is the latter then, since $\{\varepsilon<\kappa: h^\Bc_t(\varepsilon)\leq h^\Bb_{f(t)}(\varepsilon)\}\in \mathcal{D}$, we conclude that $\{\varepsilon<\kappa:h^\Bc_t(\varepsilon)<h^\Ba_{gf(t)}(\varepsilon)\}\in \mathcal{D}$.

Assume that $\{\varepsilon<\kappa:\cf(h^\Ba_{gf(t)}(\varepsilon))>\chi\}\in \mathcal{D}$. In particular, since by Proposition \ref{P:approx-small-cof}(2), $h^\Ba_{gf(t)}=h^\Bb_{f(t)}$, $\{\varepsilon<\kappa:\cf(h^\Bb_{f(t)}(\varepsilon))>\chi\}\in \mathcal{D}$. Assuming that $\Omega^\Bc_{t}$ is not trivial, the by Proposition \ref{P:approx-large-cof}(1), $\{\varepsilon<\kappa: h^\Bc_t(\varepsilon)<h^\Bb_{f(t)}(\varepsilon)\}\in \mathcal{D}$. Since $\{\varepsilon<\kappa: h^\Bb_{f(t)}(\varepsilon)\leq h^\Ba_{gf(t)}(\varepsilon)\}\in \mathcal{D}$ we conclude that $\{\varepsilon<\kappa:h^\Bc_t(\varepsilon)<h^\Ba_{gf(t)}(\varepsilon)\}\in \mathcal{D}$, as needed.

The moreover part follows immediately from the construction.
\end{proof}

\begin{lemma}\label{L:approx-limit}
Let $\delta<\sigma$ be a limit ordinal and $\langle \Ba_\alpha:\alpha<\delta\rangle$ a sequence of approximations. Assume that for $\alpha<\beta<\delta$, $\Ba_\alpha\trianglelefteq_{g_{\alpha,\beta}}\Ba_\beta$ and that for $\alpha<\beta<\gamma<\delta$, $g_{\alpha,\beta}\circ g_{\beta,\gamma}=g_{\alpha,\gamma}$. Then the inverse limit exists, i.e. there are $(\Ba_\delta, \langle g_{\alpha,\delta}: \alpha<\delta\rangle)$ such that $\Ba_\alpha\trianglelefteq_{g_{\alpha,\delta}}\Ba_\delta$ and for $\alpha<\beta<\delta$, $g_{\alpha,\beta}\circ g_{\beta,\delta}=g_{\alpha,\delta}$. 

In particular, if $\Ba_\alpha \triangleleft_{g_{\alpha,\beta}}\Ba_\beta$ for some $\alpha<\beta<\delta$ then $\Ba_\alpha \triangleleft_{g_{\alpha,\delta}}\Ba_\delta$.

Furthermore, for any $t\in S_{\Ba_\delta}$, if $\xi=\sup\{\zeta: \rho^{\Ba_\alpha}_{g_{\alpha,\delta}(t)}\in \chi^\zeta,\, \alpha<\delta\}$ then $\rho^{\Ba_\delta}_t\in \chi^{\xi+1}$.
\end{lemma}
\begin{proof}
For every $f\in \Omega$ let $t_f\in \prod_{\alpha<\delta}S_{\Ba_\alpha}$ be the function defined by $t_f(\alpha)=s$ if and only if $f\in \Omega^{\Ba_\alpha}_s$. Note that for $\alpha<\beta<\delta$, since for any $s\in S_{\Ba_\alpha}$, $\{\Omega^{\Ba_\beta}_t:t\in g_{\alpha,\beta}^{-1}(s)\}$ is a partition of $\Omega^{\Ba_\alpha}_s$, necessarily $g_{\alpha,\beta}(t_f(\beta))=t_f(\alpha)$.

Let $S_*=\{t_f: f\in \Omega\}$ and for any $t\in S_*$, let $\Omega_t=\{f\in \Omega: t_f=t\}$. Clearly, it is a partition of $\Omega$. Furthermore, note that if $t_1,t_2\in S_*$ and $\alpha<\delta$ is such that $t_1(\alpha)=t_2(\alpha)$ then $t_1(\alpha^\prime)=t_2(\alpha^\prime)$ for any $\alpha^\prime \leq\alpha$. 

Let $S_0=\{t\in S_*: (\exists \alpha <\delta) (\Omega^{\Ba_\alpha}_{t(\alpha)} \text{ is trivial})\}$ and $S_1=S_*\setminus S_0$. 

For any $t\in S_1$ and $\varepsilon<\kappa$, let $A_{t,\varepsilon}=\{h^{\Ba_\alpha}_{t(\alpha)}(\varepsilon):\alpha<\delta\}$. Obviously, $1\leq |A_{t,\varepsilon}|\leq |\delta|<\sigma\leq\chi$.

For every $(t,h/\mathcal{D})\in S_1\times\prod_{\varepsilon<\kappa} A_{t,\varepsilon}/\mathcal{D} $ let $\Omega_{(t,h/\mathcal{D})}=\{f\in \Omega_t: \forall^\mathcal{D} \varepsilon<\kappa,\, h(\varepsilon)=\min \{x\in A_{t,\varepsilon}:\Rg(f\restriction J_\varepsilon)\subseteq x\}\}$. Let $S_{\Ba_{\delta}}=\{(t,h/\mathcal{D})\in S_1\times\prod_{\varepsilon<\kappa} A_{t,\varepsilon}/\mathcal{D}: \Omega_{(t,h/\mathcal{D})}\neq \emptyset\}\cup S_0$. For $t\in S_0$, set $\Omega^{\Ba_\delta}_t=\Omega^{\Ba_\alpha}_{t(\alpha)}$ and $h^{\Ba_\delta}_t=h^{\Ba_\alpha}_{t(\alpha)}$, where $\alpha<\delta$ is minimal such that $\Omega^{\Ba_\alpha}_{t(\alpha)}$ is trivial. For $(t,h/\mathcal{D})\in S_{\Ba_\delta}\setminus S_0$, set $\Omega^{\Ba_\delta}_{(t,h/\mathcal{D})}=\Omega_{(t,h/\mathcal{D})}$. 

Note that if $t\in S_1$ then $\{\varepsilon<\kappa: \exists x\in A_{t,\varepsilon}( \Rg(f\restriction J_\varepsilon)\subseteq x)\}\in \mathcal{D}$ because this set contains $\{\varepsilon<\kappa: \Rg(f\restriction J_\varepsilon)\subseteq h^{\Ba_0}_{t(0)}(\varepsilon)\}$, which is $\mathcal{D}$ since $\Ba_0$ is an approximation and $f\in \Omega^{\Ba_0}_{t(0)}$. Thus for any $t\in S_1$ and for every $f\in \Omega_t$ there is a unique $h/\mathcal{D}\in \prod_{\varepsilon<\kappa}A_{t,\varepsilon}/\mathcal{D}$ such that $f\in \Omega_{(t,h/\mathcal{D})}^{\Ba_\delta}$. We choose $h^{\Ba_\delta}_{(t,h/\mathcal{D})}$ to be any representative of the class $h/\mathcal{D}$.

For any $t\in S_0$ and $\alpha<\delta$, $g_{\alpha,\delta}(t)=t(\alpha)$, and for every $(t,h/\mathcal{D})\in S_{\Ba_\delta}\setminus S_0$  and $\alpha<\delta$, $g_{\alpha,\delta}((t,h/\mathcal{D}))=t(\alpha)$. Note that it already follows that for $\alpha<\beta<\delta$, $g_{\alpha,\beta}\circ g_{\beta,\delta}=g_{\alpha,\delta}$.

For any $t\in S_0$ let \[\rho^{\Ba_\delta}_t=\bigcup\{\rho_{t(\alpha)}^{\Ba_\alpha}:\alpha<\delta\}^\frown \langle 0\rangle.\] 

Now let $(t,h/\mathcal{D})\in S_{\Ba_\delta}\setminus S_0$. Since $\chi^{\kappa}=\chi$, there exists some $\gamma_t\leq \chi$ and an enumeration of $\prod_{\varepsilon<\kappa} A_{t,\varepsilon}/\mathcal{D}$ as  $\langle h_{t,\gamma}/\mathcal{D}: \gamma<\gamma_t\rangle$. Now for $(t,h/\mathcal{D})\in S_{\Ba_\delta}$ set \[{\rho^{\Ba_\delta}_{(t,h/\mathcal{D})}=\bigcup\{\rho_{t(\alpha)}^{\Ba_\alpha}:\alpha<\delta\}}^\frown \langle\gamma\rangle,\] where $h/\mathcal{D}=h_{t,\gamma}/\mathcal{D}$. Assume that $t=t_f$ for some $f\in \Omega$ and let $\alpha<\beta<\delta$. Since $g_{\alpha,\beta}(t(\beta))=t(\alpha)$, $\rho^{\Ba_\alpha}_{t(\alpha)}$ is an initial segment of $\rho^{\Ba_\beta}_{t(\beta)}$. This implies that $\rho^{\Ba_\delta}_{(t,h/\mathcal{D})}\in \chi^{\xi+1}$, where $\xi=\sup\{\zeta: \rho^{\Ba_\alpha}_{g_{\alpha,\delta}(t)}\in \chi^\zeta,\, \alpha<\delta\}$.

We check that $\Ba_\delta$ is an approximation. Item (1) of Definition \ref{D:def of approximation} follows from the definition of $\Omega^{\Ba_\delta}_{t}$ and the choice of $h^{\Ba_\delta}_t$, for $t\in S_{\Ba_\delta}$. 

We show item (2). Let $(t_1,h_1/\mathcal{D})\neq (t_2,h_2/\mathcal{D})\in S_{\Ba_\delta}\setminus S_0$, assume that $\rho^{\Ba_\delta}_{(t_1,h_1/\mathcal{D})}=\rho^{\Ba_\delta}_{(t_2,h_2/\mathcal{D})}$ and let $f_1\in \Omega^{\Ba_\delta}_{(t_1,h_1/\mathcal{D})},\, f_2\in \Omega^{\Ba_\delta}_{(t_2,h_2/\mathcal{D})}$. If $t_1\neq t_2$ then there exists some $\alpha<\delta$ such that $t_1(\alpha)\neq t_2(\alpha)$. But $\rho^{\Ba_\alpha}_{t_1(\alpha)}=\rho^{\Ba_\alpha}_{t_2(\alpha)}$ and hence $f_1\not\R f_2$. Assume that $t_1=t_2$. Since $\rho^{\Ba_\delta}_{(t_1,h_1/\mathcal{D})}=\rho^{\Ba_\delta}_{(t_2,h_2/\mathcal{D})}$, $h_1/\mathcal{D}=h_2/\mathcal{D}$ which gives a contradiction. Let $(t,h/\mathcal{D})\in S_{\Ba_\delta}\setminus S_0$ and $s\in S_0$. If $s\neq t$ then the same argument as above applies. On the other it cannot be that $s=t$ be the definition of $S_0$. If $t_1\neq t_2\in S_0$ then the same arguments as above applies.

Finally, we show that $\Ba_\alpha\trianglelefteq_{g_{\alpha,\delta}}\Ba_\delta$, for $\alpha<\delta$. Items (1), (2) and (4) are straightforward. We show item (3). Let $t\in S_0$ and let $\alpha^\prime<\delta$ be minimal such that $\Omega^{\Ba_{\alpha^\prime}}_{t(\alpha^\prime)}$ is trivial. If $\alpha^\prime\leq \alpha$ then $h_t^{\Ba_\delta}=h^{\Ba_{\alpha^\prime}}_{t(\alpha^\prime)}=h^{\Ba_{\alpha}}_{t(\alpha)}$. If $\alpha<\alpha^\prime$ then $\{\varepsilon<\kappa: h^{\Ba_\delta}_t(\varepsilon)\leq h^{\Ba_\alpha}_{t(\alpha)}(\varepsilon)\}\in \mathcal{D}$ because $h_t^{\Ba_\delta}=h^{\Ba_{\alpha^\prime}}_{t(\alpha^\prime)}$. Now let $(t,h/\mathcal{D})\in S_{\Ba_\delta}\setminus S_0$. Since $\Omega^{\Ba_\delta}_{(t,h/\mathcal{D})}$ is non-empty, we may choose some function $f\in \Omega^{\Ba_\delta}_{(t,h/\mathcal{D})}$. On the one hand, since $f\in \Omega^{\Ba_\alpha}_{t(\alpha)}$ and since $\Ba_\alpha$ is an approximation, $\{\varepsilon<\kappa: \Rg(f\restriction J_\varepsilon)\subseteq h^{\Ba_\alpha}_{t(\alpha)}(\varepsilon)\}\in \mathcal{D}$. On the other hand, since $f\in \Omega^{\Ba_\delta}_{(t,h/\mathcal{D})}$, $\{\varepsilon<\kappa: h(\varepsilon)=\min \{x\in A_{t,\varepsilon}:\Rg(f\restriction J_\varepsilon)\subseteq x\}\}\in \mathcal{D}$. Combining these observation with the fact that $h^{\Ba_\alpha}_{t(\alpha)}\in A_{t,\varepsilon}$, it follows that $\{\varepsilon<\kappa: h(\varepsilon)\leq h^{\Ba_\alpha}_{t(\alpha)}(\varepsilon)\}\in \mathcal{D}$ since it contains the intersection of both sets.
\end{proof}

\begin{conclusion}\label{Con:coloring-pcf}
There exists a function $c:\Omega\to \chi$ satisfying that if $f_1,f_2\in \Omega$ and $f_1 \R f_2$ then $c(f_1)\neq c(f_2)$.
\end{conclusion}
\begin{proof}
We define $(\Ba_\xi, \langle g_{\zeta,\xi} : \zeta<\xi\rangle)$ satisfying that $\Ba_\zeta\triangleleft_{g_{\zeta,\xi}}\Ba_\xi$ (for $\zeta<\xi$), by induction on $\xi<\sigma=(2^\kappa)^+$.
\begin{list}{•}{}
\item If $\xi=0$ then $\Omega^{\Ba_0}=\Omega$, $S_a=\{0\}$, $\Omega_{0}^{\Ba_0}=\Omega^{\Ba_0}$, $\rho^{\Ba_0}_0=\emptyset$ and $h_0^{\Ba_0}\in \mathcal{H}$ be the constant function $\theta$.
\item If $\xi=\alpha +1$ for some $\alpha<\xi$ then let $\Ba_{\alpha}\triangleleft_{g_{\alpha,\xi}}\Ba_\xi$ be the approximation supplied by Proposition \ref{P:approximation,proper-successor}. For any $\zeta\leq \alpha$ we define $g_{\zeta,\xi}=g_{\zeta,\alpha}\circ g_{\alpha,\xi}$.
\item If $\xi$ is a limit ordinal we apply Lemma \ref{L:approx-limit}. 
\end{list}

It follows by induction, and using Proposition \ref{P:approximation,proper-successor}, Lemma \ref{L:approx-limit}, that for $\alpha<\beta<\sigma$ and $t\in S_{\Ba_\alpha}$ and $s\in S_{\Ba_\beta}$, if $\rho^{\Ba_\alpha}_t\in \chi^{\lambda_1}$ and $\rho^{\Ba_\beta}_s\in \chi^{\lambda_2}$ then $\lambda_1<\lambda_2$ and hence $\rho^{\Ba_\alpha}_t\neq \rho^{\Ba_\beta}_s$.

\begin{claim}
\phantomsection
\begin{enumerate}
\item For any $\rho\in \chi^{<\sigma}$, \[\Omega_\rho:=\bigcup\{\Omega^{\Ba_\xi}_s: \xi<\sigma,\, s\in S_{\Ba_\xi},\, \Omega^{\Ba_\xi}_s \text{ is trivial and } \rho^{\Ba_\xi}_s=\rho\}\] is trivial. 
\item $\Omega=\bigcup_{\rho\in \chi^{<\sigma}}\Omega_\rho$.
\end{enumerate}
\end{claim}
\begin{claimproof}
$(1)$ Let $\rho\in \chi^{<\sigma}$ and $f_1,f_2\in \Omega_\rho$. By definition, there exist $s_1\in S_{\Ba_{\xi_1}}$ and $s_2\in S_{\Ba_{\xi_2}}$ such that $f_1\in \Omega^{\Ba_{\xi_1}}_{s_1}$ and $f_2\in \Omega^{\Ba_{\xi_2}}_{s_2}$. As noted above, since $\rho^{\Ba_\xi}_{s_1}=\rho^{\Ba_{\xi}}_{s_2}$, necessarily, $\xi=\xi_1=\xi_2$.  If $s_1=s_2$ then $f_1\not\R f_2$ since $\Omega^{\Ba_{\xi}}_{s_1}=\Omega^{\Ba_{\xi}}_{s_2}$ is trivial. If $s_1\neq s_2$ then by the definition of an approximation, since $\rho^{\Ba_\xi}_{s_1}=\rho^{\Ba_{\xi}}_{s_2}$, $f_1\not\R f_2$.

$(2)$ Assume there exists some $f\in \Omega\setminus \bigcup_{\rho\in \chi^{<\sigma}}\Omega_\rho$.
We construct a a sequence of function $\langle h_\xi\in \mathcal{H}: \xi<\sigma \rangle$ satisfying that for any $\alpha<\beta<\sigma$, $h_\beta<_\mathcal{D} h_\alpha$.

For any $\xi<\sigma$ let $h_\xi=h^{\Ba_\xi}_t$ for the unique $t\in S_{\Ba_\xi}$ such that $f\in \Omega^{\Ba_\xi}_t$. By assumption, $\Omega^{\Ba_\xi}_t$ is non-trivial for any such $\xi$ (otherwise $f\in \Omega_\rho$ for $\rho=\rho_t^{\Ba_\xi}$). For any $\alpha<\beta<\sigma$, since $\Ba_\alpha\triangleleft_{g_{\alpha,\beta}} \Ba_\beta$, $h_\beta<_\mathcal{D} h_\alpha$.
%

We color pairs of function $\{(h_\alpha,h_\beta): \alpha<\beta<\sigma\}$ by $\kappa$ colors, by setting that $(h_\alpha,h_\beta)$ has color $\varepsilon_{\alpha,\beta}<\kappa$ if $\varepsilon_{\alpha,\beta}$ is the minimal $\varepsilon$ for which $h_\beta(\varepsilon)<h_\alpha(\varepsilon)$. We know that such an $\varepsilon$ exists, since $h_\beta<_\mathcal{D} h_\alpha$. By Erd\"os-Rado there exists a subset $A\subseteq \sigma$ of cardinality $\kappa^+$ and $\varepsilon<\kappa$ such that for every $\alpha<\beta\in A$, $h_\beta(\varepsilon)<h_\alpha(\varepsilon)$. This contradicts the fact that the ordinals are well-ordered.
\end{claimproof}

Recalling that $\chi^{<\sigma}=\chi$ (as cardinals), we may now define $c:\Omega\to \chi$ by choosing for every $f\in \Omega$ some $\rho\in \chi^{<\sigma}$ such that $f\in \Omega_\rho$ and setting $c(f)=\rho$. 
\end{proof}

\section{Conclusion: stable graphs}\label{s:conclusion}

We combine the results of the previous sections to conclude.

\begin{theorem}\label{T:Taylor for infinitary EM-models}
Let $\mathcal{L}$ be a first order language containing a binary relation $E$. Let $T$ be an $\mathcal{L}$-theory specifying that $E$ is a symmetric and irreflexive relation. Let $G=(V;E,\dots)\models T$ be an infinitary EM-model based on $(\alpha,\theta)$, where $\alpha\in \kappa^{U}$ for some set $U$, $\kappa\geq \aleph_0$ a cardinal and $\theta$ an ordinal with $\kappa<\theta$. Let $\kap>2^{2^{<(\kappa+\aleph_1)}}+|T|\cdot |U|$ be a regular cardinal.

If $\chi(G)\geq \kap$ then $G$ contains all finite subgraphs of $\Sh_n(\omega)$ for some $n\in \mathbb{N}$.
\end{theorem}
\begin{proof}

By Lemma \ref{L:underlying set of Inf-EM is Inf-EM} there exists some $(\widehat \alpha,\theta)$-indiscernible sequence $b=\langle b_{i,\eta}:i\in \widehat{U},\, \eta\in (\theta^{\U{\widehat{\alpha}_i}})_<\rangle$ whose underlying set is $V$, where $\widehat \alpha\in \kappa^{\widehat U}$ and  $\widehat U$ is a set such that $|\widehat{U}|\leq |T|\cdot |U|\cdot \kappa^{<\kappa}.$ 

Let $B=\{(i,\eta):i\in \widehat U,\, \eta\in (\theta^{\U{\widehat{\alpha}_i}})_<\}$ and $R=\{((i_1,\eta_1),(i_2,\eta_2)): b_{i_1,\eta_1} \E b_{i_2,\eta_2}\}$. Since $(i,\eta)\mapsto b_{i,\eta}$ is surjective and $((i_1,\eta_1),(i_2,\eta_2))\in R \iff (b_{i_1,\eta_1} b_{i_2,\eta_2})\in E$, $\chi(B,R)=\chi(G)\geq \kap$ (by Fact \ref{F:basic prop}(4)). Moreover, by Fact \ref{F:homomorphism is enough} it is enough to prove the conclusion for the graph $(B,R)$.

For any $i\in \widehat{U}$ let $B_i=\{(i,\eta):\eta\in (\theta^{\U{\alpha_i}})_<\}$. By Fact \ref{F:basic prop}(1), since $B=\bigcup_{i\in \widehat{U}} B_i$, $\kap\leq \chi(B,R)\leq \sum_{i\in \widehat{U}}\chi(B_i,R\restriction B_i)$. By definition\footnote{As $2^{<\kappa}=\sup\{2^\mu:\mu<\kappa\}$, if $2^{<\kappa}<\kappa$ then $2^{2^{<\kappa}}\leq 2^{<\kappa}$.} $\kappa\leq 2^{<\kappa}$ which implies $\kappa^{<\kappa}\leq \kappa^{\kappa}=2^{\kappa}\leq 2^{2^{<\kappa}}$ and thus $\kap>|\widehat U|$. Since $\kap$ is a regular cardinal there exists $i\in \widehat U$ with $\chi(B_i,R\restriction B_i)\geq \kap$. As a result, it is enough to prove the conclusion for the graph $((\theta^{\U{\widehat{\alpha}_i}})_<,S)$, where $S=\{(\eta_1,\eta_2):(i,\eta_1)\R (i,\eta_2)\}$.

For $P=\{\otp(\bar a,\bar b):(\bar a,\bar b)\in S\}$, by $(\widehat{\alpha},\theta)$-indiscernibility $S=\bigcup_{p\in P} \{(\bar c,\bar d)\in ((\theta^{\U{\widehat{\alpha}_i}})_<)^2 :\otp(\bar c,\bar d)=p\vee \otp(\bar d,\bar c)=p\}$. 

By Fact \ref{F:basic prop}(2), 
\[\kap\leq \chi((\theta^{\U{\widehat{\alpha}_i}})_<,S))\leq \prod_{p\in P}\chi ((\theta^{\U{\widehat{\alpha}_i}})_<,P_p),\] where $P_p=\{(\bar c,\bar d)\in ((\theta^{\U{\widehat{\alpha}_i}})_<)^2 :\otp(\bar c,\bar d)=p\vee \otp(\bar d,\bar c)=p\}$. 
Assume towards a contradiction that $\chi ((\theta^{\U{\widehat{\alpha}_i}})_<,P_p)\leq \beth_2(\aleph_0)$ for all $p\in P$. Hence $\kap\leq \beth_2(\aleph_0)^{2^{|\alpha_i|+\aleph_0}}\leq \beth_2(|\alpha_i|+\aleph_0)$. Since $|\alpha_i|+\aleph_0<\kappa+\aleph_1$ and $\kap>2^{2^{<(\kappa+\aleph_1)}}$, we derive a contradiction.

Consequently, there exists $p\in P$ with $\chi ((\theta^{\U{\widehat{\alpha}_i}})_<,P_p)>\beth_2(\aleph_0)$ and we may conclude by Theorem \ref{T:shift graphs in order-type-graphs}.
\end{proof}

\begin{corollary}\label{C:main corollary}
Let $G=(V,E)$ be a stable graph. If $\chi(G)>\beth_2(\aleph_0)$ then $G$ contains all finite subgraphs of $\Sh_n(\omega)$ for some $n\in \mathbb{N}$.
\end{corollary}
\begin{proof}
Let $G=(V,E)$ be stable graph, $T=Th(G)$ and $T^{sk}$ be a complete expansion of $T$ with definable Skolem functions in the language $E\in \mathcal{L}^{sk}$.  

We apply Theorem \ref{T:existence of gen em model in stable} with $\kappa=\aleph_1$, $\mu=2^{\aleph_1}$ and $\lambda=2^{\max\{\mu,|G|\}}$. We get an infinitary EM-model $\mathcal{G}^{sk}\models T^{sk}$ based on $(\alpha,\lambda)$, where $\alpha \in \kappa^U$ for some set $U$ of cardinality at most $\mu$, such that $\mathcal{G}=\mathcal{G}^{sk}\restriction \{E\}$ is saturated of cardinality $\lambda$. Since $\mathcal{G}$ is saturated of cardinality $>|G|$, we may embed $G$ as an elementary substructure of $\mathcal{G}$. Since $\chi(\mathcal{G})\geq \chi(G)>\beth_2(\aleph_0)$ and the conclusion is an elementary property, it is enough to show it for $\mathcal{G}$. 

Since $2^{2^{<(\kappa+\aleph_1)}}+|T|+|U|\leq 2^{2^{\aleph_0}}+\aleph_0+\mu\leq \beth_2(\aleph_0)+2^{\aleph_1}=  \beth_2(\aleph_0)$, Theorem \ref{T:Taylor for infinitary EM-models} applies with $\theta=\lambda$ and $\kap=(\beth_2(\aleph_0))^+$.
\end{proof}

\bibliographystyle{alpha}
\bibliography{1211}

\end{document}